\documentclass[a4paper,12pt,reqno]{amsart}

\usepackage[english]{babel}
\usepackage[T1]{fontenc}
\usepackage[utf8]{inputenc}

\usepackage{tgpagella}
\usepackage{mathpazo}

\usepackage{microtype}

\usepackage{mathtools}        
\usepackage{amssymb}
\usepackage{amsthm}
\usepackage[mathscr]{eucal}
\usepackage{bbm}

\usepackage{graphicx}
\usepackage{subcaption}       

\usepackage{setspace}

\usepackage[dvipsnames]{xcolor}
\usepackage[draft]{changes}   

\usepackage[
  colorlinks = true,
  linkcolor  = blue,
  citecolor  = blue,
  urlcolor   = blue
]{hyperref}

\numberwithin{equation}{section}

\newtheorem{theorem}{Theorem}[section]
\newtheorem{proposition}[theorem]{Proposition}

\newtheorem{corollary}[theorem]{Corollary}
\newtheorem{lemma}[theorem]{Lemma}

\newtheorem{assumption}{Assumption}

{\theoremstyle{definition}
{
\newtheorem{remark}[theorem]{Remark}

\newtheorem{defn}[theorem]{Definition}
}}

\newcommand{\cal}{\mathcal}

\newcommand{\BB}{{\cal B}}

\newcommand{\NN}{{\cal N}}

\newcommand{\XX}{{\cal X}}

\newcommand{\ZZ}{{\cal Z}}

\newcommand{\Nn}{{\mathbb{N}}}

\newcommand{\Pp}{{\mathbb{P}}}

\newcommand{\Rr}{{\mathbb{R}}}

\newcommand{\Tt}{{\mathbb{T}}}

\newcommand{\Zz}{{\mathbb{Z}}}

\def\SL{\operatorname{SL}}

\def\tr{\operatorname{tr}}

\def\supp{\operatorname{supp}}

\begin{document}
\title
[Lyapunov exponents and rigidity]{Lyapunov exponents and rigidity in random billiards and standard maps}

\date{\today}

\author[Del Magno]{Gianluigi Del Magno}
\address{Dipartimento di Matematica, Universit\`a di Pisa \\
Largo Bruno Pontecorvo 5, 56127 Pisa, Italy}
\email{gianluigi.delmagno@unipi.it}

\author[Lopes Dias]{Jo\~{a}o Lopes Dias}
\address{Universidade de Lisboa, ISEG Lisbon School of Economics \& Management, ISEG Research \\
Rua do Quelhas 6, 1200-781 Lisboa, Portugal}
\email{jldias@iseg.ulisboa.pt}

\author[Gaiv\~ao]{Jos\'e Pedro Gaiv\~ao}
\address{Universidade de Lisboa, ISEG Lisbon School of Economics \& Management, ISEG Research \\
Rua do Quelhas 6, 1200-781 Lisboa, Portugal}
\email{jpgaivao@iseg.ulisboa.pt}

\begin{abstract}
We study random dynamical systems generated by measure-preserving maps, possibly with singularities. For this class of systems, we establish an invariance principle: if all Lyapunov exponents coincide almost everywhere, then there exists an invariant measurable family of probability measures on the projective tangent bundle for the projective cocycle induced by the derivative. We apply this principle to random additive perturbations of two classes of maps: billiard maps for strictly convex tables on surfaces of constant curvature and generalized standard maps. In both settings, we obtain rigidity results: the Lyapunov exponents vanish almost everywhere precisely for billiards in geodesic disks and generalized standard maps with constant potential; in all other cases, they are nonzero almost everywhere.

\end{abstract}

\maketitle
\tableofcontents

\section{Introduction}
\label{sec:introduction}
Numerical simulations indicate that many volume-preserving maps exhibit a mixed phase space, with coexisting regions of regular and chaotic dynamics; see~\cite{Strelcyn1991} and the references therein. This behavior occurs in a broad class of systems, including area-preserving twist maps and billiards in convex tables, and in many examples suggests the presence of positive Lyapunov exponents on a set of positive volume. Despite substantial numerical and heuristic evidence, establishing this property rigorously for deterministic conservative systems remains a difficult problem.

This difficulty motivates the study of conservative systems in the presence of noise. Random perturbations are expected to weaken invariant structures and improve ergodic properties, thereby promoting exponential instability. In this paper, we study random additive perturbations of conservative maps and ask whether they have non-zero Lyapunov exponents almost everywhere. When the exponents vanish, we ask whether the underlying deterministic system has special geometric or algebraic properties. This is the rigidity problem at the center of the paper. 

A key tool in our analysis is an invariance principle for random dynamical systems generated by maps with singularities. This terminology refers to maps that are $C^1$ diffeomorphisms on an open set of full measure, as in the framework of~\cite{Katok:1986tg}.

Let $M$ be a Riemannian manifold, let $\nu$ be a probability measure on a space $X$, and let
\[
f_x\colon M\to M,
\qquad x\in X,
\]
be a measurable family of maps preserving a common probability measure $m$ on $M$. Set
\[
\Omega=X^{\Nn},
\qquad
\rho_\nu=\nu^{\Nn},
\]
write $\omega=(\omega_0,\omega_1,\ldots) \in \Omega$, and denote by $\sigma\colon\Omega\to\Omega$ the left shift. The associated random dynamical system is the skew product
\[
F\colon\Omega\times M\to\Omega\times M,
\qquad
F(\omega,y)
=
\bigl(\sigma(\omega),f_{\omega_0}(y)\bigr).
\]
The random compositions generated by the family $(f_x)_{x\in X}$ are
\[
F_\omega^0=\operatorname{id}_M,
\qquad
F_\omega^n
=
f_{\omega_{n-1}}\circ\cdots\circ f_{\omega_0},
\qquad n\geq1.
\]

We assume that each $f_x$ restricts to a $C^1$ diffeomorphism on an open set of full $m$-measure and that the derivative cocycle satisfies suitable integrability conditions. We denote by $\lambda_F^+$ and $\lambda_F^-$ the maximal and minimal Lyapunov exponents of $F$, respectively. The precise assumptions and definitions are given in Section~\ref{sec:random maps}. Under these hypotheses, we prove the following result.

\begin{theorem}
Suppose that $\lambda_F^-(\omega,y) = \lambda_F^+(\omega,y)$ for $(\rho_\nu\times m)$-almost every $(\omega,y)\in\Omega\times M$. Then there exists a probability measure $\eta$ on the projective tangent bundle $PM$ such that $\pi_*\eta=m$, where $\pi\colon PM\to M$ is the bundle projection, and
\[
(Df_x)_*\eta=\eta
\qquad \text{for } \nu \text{-a.e. } x\in X.
\]
Equivalently, if $\{\eta_y:y\in M\}$ denotes the disintegration of $\eta$ over $m$, then
\[
Df_x(y)_*\eta_y
=
\eta_{f_x(y)}
\]
for $(\nu\times m)$-almost every $(x,y)\in X\times M$.
\end{theorem}

Following Avila and Viana~\cite{Avila:2010aa}, we refer to this result as an invariance principle. It provides the basic rigidity mechanism used throughout the paper: coincidence of the extremal Lyapunov exponents forces the existence of a non-random invariant structure on the projective tangent bundle. The result is obtained by applying a theorem of Ledrappier~\cite{Ledrappier:1986ue} for linear cocycles over stationary Markov chains to the derivative cocycle of the random system. The underlying principle goes back to Furstenberg's theory of random matrix products~\cite{Furstenberg:1963uq}. Analogous invariance principles for random diffeomorphisms were proved by Carverhill~\cite{carverhill1987furstenberg} and Baxendale~\cite{Baxendale:1989aa}.

We next specialize the invariance principle to random additive perturbations of conservative maps on the torus. Let
\[
g\colon\Tt^d\to\Tt^d
\]
be a map that preserves the normalized Haar measure $m$ and satisfies the regularity and integrability assumptions introduced in Definition~\ref{def:toral_map}. In this setting, $X=\Tt^d$ and $\Omega=(\Tt^d)^{\Nn}$. For each $x\in\Tt^d$, define
\[
f_x(y)=g(y)+x.
\]
Given an i.i.d.\ sequence $\omega=(\omega_n)_{n\geq0} \in \Omega$ of $\Tt^d$-valued random variables with common distribution $\nu$, consider the corresponding random compositions $F_\omega^n$. 

For these systems, the sum of the Lyapunov exponents vanishes almost everywhere. Consequently, $\lambda_F^-=\lambda_F^+$ almost everywhere if and only if $\lambda_F^+=0$ almost everywhere. The invariance principle therefore has the following consequence. We write $\SL^{\pm}(d,\Rr)$ for the set of real $d \times d$ matrices with determinant $\pm 1$.

\begin{theorem}
Let $F$ be a $\nu$-random additive perturbation of a toral map
$g\colon\Tt^d\to\Tt^d$, and let $N\subset\Tt^d$ be an open set of full
$m$-measure such that $g|_N\colon N\to g(N)$ is a $C^1$ diffeomorphism. Suppose that $\lambda_F^+(\omega,y)=0$ for $(\rho_\nu\times m)$-almost every
$(\omega,y)\in\Omega\times\Tt^d$. Then there exists a measurable family
of probability measures $\{\eta_y:y\in\Tt^d\}$ on $\mathbb P^{d-1}$ such
that
\[
Dg(y)_*\eta_y=\eta_{g(y)+x}
\]
for $m$-almost every $y\in\Tt^d$ and $\nu$-almost every $x\in\Tt^d$. If, in addition, $\nu$ is absolutely continuous with respect to $m$, then
there exists a probability measure $\eta$ on $\mathbb P^{d-1}$ such that
\[
h_*\eta=\eta
\]
for every $h\in\operatorname{supp}\bigl((Dg)_*(m|_N)\bigr)$. Hence, the support of $(Dg)_*(m|_N)$ is either contained in a
compact subgroup of $\operatorname{SL}^{\pm}(d,\Rr)$ or preserves a finite family of proper subspaces of $\Rr^d$.
\end{theorem}

Thus, for random additive perturbations with absolutely continuous noise, vanishing Lyapunov exponents force the derivatives of the underlying deterministic map to preserve a common projective probability measure. 

We apply this rigidity mechanism to two fundamental classes of conservative dynamical systems: billiards in convex tables and generalized standard maps. For random additive perturbations of convex billiards, we prove that the Lyapunov exponents vanish almost everywhere precisely if and only if the table is a geodesic disk. For random additive perturbations of generalized standard maps, they vanish almost everywhere if and only if the potential is constant.

We begin with billiards on surfaces of constant curvature. Let $S$ be one of the standard Riemannian surfaces of constant curvature: the Euclidean plane, the sphere, or the hyperbolic plane. Let $D\subset S$ be a strictly convex domain whose boundary $\partial D$ is a $C^2$ simple closed curve of positive geodesic curvature and unit length. We refer to such a domain $D$ as a convex table. A billiard in $D$ describes the motion of a point particle travelling along geodesics of $S$ inside $D$ and undergoing elastic reflections at $\partial D$.

The dynamics is encoded by the billiard map $\Phi\colon V\to V$, a twist map on the collision space $V=S^1\times[-1,1]$. Each point $(s,r)\in V$ represents a collision with the boundary: $s$ is the arc-length parameter of the collision point on $\partial D$, and $r=-\cos\theta$, where $\theta\in[0,\pi]$ is the reflection angle measured with respect to the tangent to $\partial D$ at $s$. The map $\Phi$ sends each collision to the next collision along the trajectory. After the identification described in Section~\ref{sec:billiards}, the billiard map induces a map $T\colon\Tt^2\to\Tt^2$ preserving the normalized Haar measure $m$.

Despite extensive study, it remains an open problem whether billiards in strictly convex tables with smooth boundary can exhibit positive metric entropy with respect to $m$. All known examples of billiards with this property involve tables that are nor strictly convex nor smooth; see, for instance, \cite{Bunimovich:79,Wojtkowski:86,Donnay:91,Markarian:94,Gutkin:99}. Although generic convex billiards have positive topological entropy \cite{BessaDelMagnoLopesDiasGaivaoTorres24,Cheng2004400,ZHANG2017793}, this does not by itself provide examples with positive metric entropy with respect to $m$.

The following theorem establishes a sharp rigidity dichotomy for random additive perturbations of convex billiards, distinguishing geodesic disks from all other convex tables. Here $T$ denotes the billiard map on $\Tt^2$, the maps
\[
f_x(y)=T(y)+x, \qquad x \in \Tt^2,
\]
are its additive perturbations, and $F_\omega^n$ denotes the corresponding random composition. Additive perturbations of the billiard map admit a natural geometric interpretation. Given $x=(\bar s,\bar r)\in\Tt^2$ and an initial collision $(s_0,r_0)$, one first computes the next collision $(s_1,r_1)=T(s_0,r_0)$
and then translates it by $x$ in phase space, shifting the arc-length coordinate by $\bar s$ and the $r$-coordinate by $\bar r$.

\begin{theorem}
\label{thm:main}
Let $D$ be a convex billiard table on a surface $S$ of constant curvature. Assume that the common distribution $\nu$ of the random variables $(\omega_n)_{n\geq0}$ is absolutely continuous with respect to $m$. Then the Lyapunov exponents of the random compositions $F_\omega^n$ are non-zero almost everywhere if $D$ is not a geodesic disk, and vanish almost everywhere otherwise.
\end{theorem}

This theorem reflects the exceptional symmetry of circular billiards on surfaces of constant curvature. If $D$ is a geodesic disk, then the phase space is foliated by non-contractible invariant circles \cite{Bialy:1993ty,dosSantos:2017}, a property commonly referred to as total integrability. Under an additive perturbation, this foliation remains globally invariant: individual circles need not be invariant, but each circle is mapped onto another circle of the same foliation. Consequently, the derivative cocycle has no exponential growth, and the Lyapunov exponents vanish almost everywhere.

The converse is the rigidity direction of the theorem. Suppose that the Lyapunov exponents vanish. The invariance principle and its consequences for additive perturbations yield a common invariant projective probability measure. The factorization of the derivative cocycle, together with its behavior near tangential collisions, forces this measure to be the Dirac mass at the horizontal direction. Hence the horizontal direction is invariant under the derivative of the billiard map. This implies that the billiard foliation by horizontal circles is invariant and ultimately forces the boundary of the table to have constant geodesic curvature. Therefore $D$ is a geodesic disk.

Theorem~\ref{thm:main} considerably broadens the scope of the result of Zhang and Nguyen~\cite{NguyenZhang2026}, who proved the existence of non-zero Lyapunov exponents for random additive perturbations of billiards in non-circular elliptical tables and lemon-shaped tables in the Euclidean plane. It also provides a new rigidity characterization of circular billiards among convex billiards on surfaces of constant curvature (cf.~\cite{Bialy:1993ty} and~\cite{CzudekDeSimoiGadPoon26}). 

The second class of conservative systems considered in this paper consists of generalized standard maps. Given a $C^1$ function
\[
V\colon\Tt^1\to\Rr,
\]
the associated generalized standard map $g_V\colon\Tt^2\to\Tt^2$ is
\[
g_V(y_1,y_2)
=
\bigl(y_1+y_2+V(y_1),\,y_2+V(y_1)\bigr)
\pmod 1.
\]
For $V(y_1)=K\sin(2\pi y_1)$, this is the standard map introduced by Chirikov~\cite{CHIRIKOV1979263}; the case $K=0$ is completely integrable.

Generalized standard maps arise naturally in Hamiltonian dynamics as models of periodically kicked systems and play a central role in the study of the transition from integrability to chaos. Nevertheless, proving the existence of non-zero Lyapunov exponents for non-integrable standard maps and their generalizations remains notoriously difficult.

The following theorem is the counterpart of Theorem~\ref{thm:main} for generalized standard maps. For each $x\in\Tt^2$, define
\[
f_x(y)=g_V(y)+x,
\]
and let $F_\omega^n$ denote the corresponding random composition. The theorem establishes the corresponding sharp rigidity dichotomy, distinguishing constant potentials from non-constant ones. In this setting, the conclusion holds both for absolutely continuous noise and for a class of degenerate noise distributions.

We call a Borel probability measure $\nu$ on $\Tt^2$ a \emph{vertical line measure} if
\[
\nu=\delta_0\times\chi,
\]
where $\delta_0$ is the Dirac measure at $0\in\Tt^1$ and
$\chi=h\,m_1$ is absolutely continuous with respect to the normalized Haar measure $m_1$ on $\Tt^1$.

\begin{theorem}
\label{thm:standard}
Let $g_V\colon\Tt^2\to\Tt^2$ be a generalized standard map. Assume that the common distribution $\nu$ of the random variables $(\omega_n)_{n\geq0}$ is either absolutely continuous with respect to $m$ or a vertical line measure. Then the Lyapunov exponents of the random compositions $F_\omega^n$ are non-zero almost everywhere if $V$ is non-constant, and vanish almost everywhere otherwise.
\end{theorem}

The derivative matrices of a non-constant generalized standard map contain both hyperbolic and elliptic elements, which excludes the alternatives compatible with an invariant projective probability measure. For vertical line noise, the random variables are grouped in pairs, reducing the problem to an additive perturbation with an absolutely continuous noise distribution.

When $\nu$ is absolutely continuous, the result follows 
from the invariance principle. If $V$ is non-constant, among the derivative matrices, there are both hyperbolic and elliptic elements, which excludes the alternatives compatible with an invariant projective probability measure. For vertical line noise, the random variables are grouped in pairs, reducing the problem to a random additive perturbation of $g_V^2$ with an absolutely continuous noise distribution.

For random standard maps, Blumenthal, Xue, and Young~\cite{BlumenthalXueYoung} obtained quantitative lower bounds for the maximal Lyapunov exponent under sufficiently strong random additive perturbations. In contrast, the present result does not require the noise to be large and gives a sharp characterization of the vanishing of the Lyapunov exponents.

The proofs of the two theorems follow a common rigidity scheme. First, we prove that the stationary measure is ergodic for the relevant random additive perturbations. This yields a dichotomy for the maximal Lyapunov exponent: it either vanishes almost everywhere or is positive almost everywhere. In the vanishing case, the invariance principle produces a measurable invariant family of projective measures. The additive noise promotes this family to a common projective measure, and the structure of the derivative cocycle turns that measure into an invariant geometric structure for the deterministic map. The final rigidity step classifies the maps carrying this structure: the billiard table must be a geodesic disk, and the potential of the generalized standard map must be constant.

The paper is organized as follows. Section~\ref{sec:random maps} establishes a general invariance principle for random piecewise smooth maps and derives several consequences needed later in the paper. Section~\ref{se:ergodicity} proves the unique ergodicity of random additive perturbations and the equidistribution of their random orbits in two general settings. These settings cover the random additive perturbations of billiard maps and generalized standard maps treated in the final sections. Section~\ref{sec:billiards} provides background on convex billiards on surfaces of constant curvature. The proof of Theorem~\ref{thm:main} is given in Section~\ref{ss:random_billiards}. Finally, Section~\ref{sec:random_standard} is devoted to random generalized standard maps and contains the proof of
Theorem~\ref{thm:standard}.

\section{An invariant principle for random maps}
\label{sec:random maps}
\subsection{Random dynamical systems} 
We model a random dynamical system as a skew-product over a shift transformation. The definition of a random dynamical system provided below is not the most general. For more general definitions, see \cite{Arnold:98,Kifer:86}.


Throughout the paper, $\Nn$ denotes the natural numbers, including zero. Let $M$ be a complete smooth Riemannian manifold of dimension $d$, equipped
with the Borel $\sigma$-algebra $\BB$. We denote by $\operatorname{vol}$ the Riemannian volume measure on $M$.

Let $(X, \XX, \nu)$ be a probability space, where $X$ is a separable complete metric space, $\XX$ is its Borel $\sigma$-algebra, and $\nu$ is a probability measure on $\XX$. Denote by $(\Omega,\Sigma,\rho_\nu)$ the probability space given by the product of countably many copies of $(X, \XX, \nu)$. Namely, $\Omega$ is the space of sequences 
\[
\Omega = X^{\Nn} = \{(\omega_n)_{n \in \Nn} \colon \omega_n \in X \text{ for all } n \in  \Nn\},
\] 
equipped with the product $\sigma$-algebra $\Sigma = \XX^{\Nn}$ and the product measure $\rho_\nu = \nu^{\Nn}$. The shift map $\sigma \colon \Omega \to \Omega$ defined by
\[(\sigma(\omega))_n = \omega_{n+1} \quad \text{for every } \omega \in \Omega,\] 
is measurable and preserves the probability measure $\rho_\nu$.

Since $X$ and $M$ are separable complete metric spaces, the same is true of $X\times M$, $\Omega$ and $\Omega\times M$. All these spaces are endowed
with their Borel $\sigma$-algebras.




Let $f \colon X \times M \to M$ be a measurable map. For every $x \in X$, define 
\[
f_x(\cdot)=f(x,\cdot) \colon M \to M.
\]

\begin{defn}
We call the skew-product $F \colon \Omega \times M \to \Omega \times M$ defined by
\[
F(\omega,y)=(\sigma(\omega),f_{\omega_0}(y)) \quad  \text{for every } (\omega,y) \in \Omega \times M,
\]  
the \emph{random dynamical system on $M$ generated by $f$ and $\nu$}.
\end{defn}

Note that $F$ depends on $\omega$ only through its 0th component $\omega_0$. For every $n \in \Nn$ and every $ \omega \in \Omega$, define $F^{n}_\omega \colon M \to M$ as follows: 
\[
F^{n}_\omega = 
\begin{cases}
\operatorname{id}_M & \text{if } n=0, \\
f_{\omega_{n-1}} \circ \cdots \circ f_{\omega_0}  & \text{if } n \ge 1.
\end{cases}
\]
The iterates of $F$ can be written as 
\[
F^{n}(\omega,y)=(\sigma^n(\omega),F^{n}_{\omega}(y)) \quad  \text{for every } (\omega,y) \in \Omega \times M.
\]

\begin{defn}
Let $y \in M$. For every $\omega \in \Omega$, the sequence $\{F^n_{\omega}(y)\}_{n \ge 0}$ is called the \emph{random orbit of $y$ associated with $\omega$}.
\end{defn}

We now formulate the standing assumptions on $f$ and $\nu$ that will be assumed throughout this paper.

\begin{assumption}\label{as:one}
There exists a Borel probability measure $m$ on $M$ that is $f_x$-invariant for every $x\in X$. Moreover, for every $x\in X$, there exists an open set $N_x\subset M$ with
$m(N_x)=1$ such that
\[
    f_x|_{N_x}\colon N_x\to f_x(N_x)\subset M
\]
is a $C^1$ diffeomorphism. Finally, the set
\[
    \NN:=\{(\omega,y)\in\Omega\times M:\ y\in N_{\omega_0}\}
\]
is measurable in $\Omega\times M$.
\end{assumption}


\begin{remark}
Assumption~\ref{as:one} allows the maps $f_x$ to have singular points, as in the case of billiard maps. We call a point $y\in M$ singular for $f_x$ if $f_x$ is not a $C^1$ local diffeomorphism at $y$. Since
$f_x|_{N_x}$ is a $C^1$ diffeomorphism, every singular point of $f_x$ belongs to $M\setminus N_x$. Moreover, the invariance of $m$ under $f_x$ for every $x\in X$ implies that $\rho_\nu\times m$ is invariant under the skew product $F$. Finally, since $m(N_x)=1$ for every $x\in X$, Fubini's theorem gives $(\rho_\nu\times m)(\NN)=1$.
\end{remark}

We fix a global measurable trivialization $\Psi\colon TM\to M\times\Rr^d$ of the tangent bundle such that for each $y\in M$, the restriction
\[
\Psi_y:=\Psi|_{T_yM}\colon T_yM\to\Rr^d
\]
is a linear isometry from $T_yM$, endowed with the inner product induced by the Riemannian metric, onto $\Rr^d$, endowed with its standard Euclidean inner product; see \cite[Lemma~4.2.4]{Arnold:98}.

Define $A\colon \Omega\times M\to \operatorname{GL}(d,\Rr)$ by
\[
A(\omega,y)
=
\begin{cases}
\Psi_{f_{\omega_0}(y)}\circ Df_{\omega_0}(y)\circ \Psi_y^{-1}
& \text{if } (\omega,y)\in \NN,\\[2mm]
I
& \text{if } (\omega,y)\in (\Omega\times M)\setminus \NN,
\end{cases}
\]
where $I$ denotes the identity matrix in $\operatorname{GL}(d,\Rr)$.

On $\NN$, the matrix $A(\omega,y)$ represents the map $Df_{\omega_0}(y)$ with respect to the isometric trivialization $\Psi$. Therefore
\[
\|A(\omega,y)\|
=
\|Df_{\omega_0}(y)\|_{y,f_{\omega_0}(y)}
\qquad\text{for all }(\omega,y)\in\NN,
\]
where the norm on the left-hand side is the Euclidean operator norm on $\Rr^d$, while the norm on the right-hand side is the Riemannian operator norm from $T_yM$ to $T_{f_{\omega_0}(y)}M$.

Since $\NN$ is measurable, $\Psi$ is measurable, and
$(\omega,y)\mapsto Df_{\omega_0}(y)$ is measurable on $\NN$, the map $A$ is measurable on $\Omega\times M$. We set $A(\omega,y)=I$ on
$(\Omega\times M)\setminus\NN$ only to obtain a measurable map defined everywhere.

Set $A^0(\omega,y)=I$. For $n\ge 1$, define
\[
A^n(\omega,y)
=
A(F^{n-1}(\omega,y))\cdots A(F(\omega,y))A(\omega,y).
\]

Note that the choice of $A$ on $(\Omega\times M)\setminus\NN$ does not affect
$A^n(\omega,y)$ on the full $\rho_\nu\times m$-measure set of points whose forward $F$-orbit stays in $\NN$.

\begin{defn}
The maximal and minimal Lyapunov exponents of $F$ at $(\omega,y) \in \Omega \times M$ are defined by
\[
\lambda_F^+(\omega,y) = \limsup_{n\to+\infty} \frac1n\log\|A^n(\omega,y)\|
\]
and
\[
\lambda_F^-(\omega,y)
= -\limsup_{n\to+\infty}
\frac1n\log\|(A^n(\omega,y))^{-1}\|.
\]
\end{defn}

\begin{assumption}\label{as:three}
$\log^+\|A\|,\ \log^+\|A^{-1}\|
    \in L^1(\rho_\nu\times m)$.
\end{assumption}

\begin{remark}
By Assumption~\ref{as:three} and the Subadditive Ergodic Theorem, the limsup in the definitions of $\lambda_F^+(\omega,y)$ and $\lambda_F^-(\omega,y)$ can be replaced by the limit for $\rho_\nu\times m$-a.e. $(\omega,y)\in\Omega\times M$. Since the Lyapunov exponent are $F$-invariant, ergodicity of
$\rho_\nu\times m$ with respect to $F$ implies that both exponents are constant $\rho_\nu\times m$-a.e.
\end{remark}

Since $(\rho_{\nu} \times m)(\NN)=1$, the set of points $(\omega,y)$ such that
\[
F^n(\omega,y) \in \NN
\qquad\text{for all } n\ge 0
\]
has full $\rho_\nu\times m$-measure. On this set, $A^n(\omega,y)$ is the matrix representation of $DF_\omega^n(y)$ with respect to $\Psi$. Hence for $\rho_\nu\times m$-a.e. $(\omega,y)$, 
\[
\lambda_F^+(\omega,y)
=
\lim_{n\to+\infty}
\frac1n\log\|DF_\omega^n(y)\|_{y,F_\omega^n(y)}
\]
and
\[
\lambda_F^-(\omega,y)
=
-\lim_{n\to+\infty}
\frac1n\log\|(DF_\omega^n(y))^{-1}\|_{F_\omega^n(y),y}.
\]

\subsection{Random additive perturbations of toral maps}
\label{random additive}

We now specialize to random dynamical systems for which both the base space
and the fiber are the $d$-dimensional torus
\[
X=M=\Tt^d=\Rr^d/\Zz^d.
\]
We endow $\Tt^d$ with its standard flat metric and denote by $m$ its normalized Haar measure.

For brevity, throughout this paper the term \emph{toral map} will be used in the specialized sense specified in the following definition. In particular, it will always refer to a map preserving $m$ and satisfying the stated regularity and integrability assumptions.

\begin{defn}
\label{def:toral_map}
A measurable map $g\colon\Tt^d\to\Tt^d$ is called a \emph{toral map} if it satisfies the following properties:
\begin{enumerate}
\item $g$ preserves the Haar measure $m$, 
\item there exists an open set $N\subset\Tt^d$ with $m(N)=1$ such that $g|_N\colon N\to g(N) \subset \Tt^d$
is a $C^1$ diffeomorphism,
\item the functions $1_N(y)\log^+\|(Dg(y))^{\pm 1}\|$  belong to $L^1(m)$, where they are understood to vanish on $\Tt^d\setminus N$ and the norm is the operator norm induced by the
standard flat metric on $\Tt^d$. 
\end{enumerate}
\end{defn}

\begin{defn}
Let $g\colon\Tt^d\to\Tt^d$ be a toral. A map
$f\colon\Tt^d\times\Tt^d\to\Tt^d$ of the form
\[
f(x,y)=\tau_x\circ g(y)=g(y)+x,
\qquad x,y\in\Tt^d,
\]
where $\tau_x$ denotes translation by $x$ on $\Tt^d$, is called an
\emph{additive perturbation} of $g$. The random dynamical system generated by such a map $f$ and by a probability measure $\nu$ on $\Tt^d$ is called a \emph{$\nu$-random additive
perturbation of $g$}. In this setting, $\Omega=(\Tt^d)^{\Nn}$ and $\rho_\nu=\nu^{\Nn}$.
\end{defn}




Random additive perturbations toral maps provide examples of systems satisfying Assumptions~\ref{as:one} and~\ref{as:three}.

\begin{lemma}
\label{le:add-pert}
Every $\nu$-random additive perturbation $F$ of a toral map satisfies Assumptions~\ref{as:one} and~\ref{as:three}. Moreover, for such a system, 
the condition $\lambda_F^-(\omega,y)=\lambda_F^+(\omega,y)$ for $\rho_\nu\times m$-a.e. $(\omega,y)$ is equivalent to the condition $\lambda_F^+(\omega,y)=0$ for $\rho_\nu\times m$-a.e. $(\omega,y)$.
\end{lemma}

\begin{proof}
Translations of $\Tt^d$ are diffeomorphisms preserving the Haar measure $m$. Hence, for every $x\in\Tt^d$, the map $f_x=\tau_x\circ g$ preserves $m$, and $f_x|_N\colon N\to f_x(N)=\tau_x(g(N))$ is a $C^1$ diffeomorphism. Moreover, taking $N_x=N$ for every $x\in\Tt^d$, we have
$\NN=\Omega\times N$, which is measurable. Thus $F$ satisfies Assumption~\ref{as:one}.

Since the differential of a translation is equal to the identity in the standard flat trivialization, we have 
\[
Df_x(y)=Dg(y) \qquad \text{for all } (x,y) \in\Tt^d \times N.
\]
Thus $\|(A(\omega,y))^{\pm 1}\|=\|(Dg(y))^{\pm 1}\|$ on $\NN$. Since $A(\omega,y)=I$ on $(\Omega \times \Tt^d) \setminus \NN$, the integrability assumptions on $(Dg)^{\pm 1}$ imply that $\log^+ \|A^{\pm 1}\| \in L^{1}(\rho_{\nu} \times m)$. Therefore $F$ satisfies Assumption~\ref{as:three}.

Furthermore, $|\det A^n(\omega,y)|=1$
for every $n\geq1$ and for $\rho_\nu\times m$-a.e.\ $(\omega,y)$. By Oseledets' Theorem,
\[
\lim_{n\to\infty}\frac1n\log|\det A^n(\omega,y)|
\]
is equal to the sum of the Lyapunov exponents of $F$, counted with multiplicity. Since the left-hand side is zero, this sum vanishes for $\rho_\nu\times m$-a.e.\ $(\omega,y)$.

If $\lambda_F^-=\lambda_F^+$ almost everywhere, then all Lyapunov exponents coincide almost everywhere, since every exponent lies between the minimal
and maximal ones. Their sum being zero, they must all vanish, and in particular $\lambda_F^+=0$ almost everywhere. The converse is immediate
from $\lambda_F^- \le \lambda_F^+$ and the vanishing of the sum of the Lyapunov exponents.
\end{proof}



\subsection{Invariance principle}
In this section, we first recall a theorem of Ledrappier that generalizes Furstenberg's classical result on zero Lyapunov exponents for products of independent unimodular matrices~\cite{Furstenberg:1963uq} to linear cocycles over Markov chains. Similar results have been obtained in \cite{Royer:1980vh,Virtser:80,Guivarch:79}. Then we apply Ledrappier's theorem to random maps.

\subsubsection{Ledrappier's Invariance Theorem}

Let $Z$ be a separable complete metric space with its Borel $\sigma$-algebra $\ZZ$. Consider a Markov chain $\{Z_n\}_{n \ge \Nn}$ with state space $(Z,\ZZ)$, transition probability $\{P_z \colon z \in Z\}$ and stationary probability measure $\zeta$. 

Let $(\Omega',\Sigma')$ be the two-sided product space $\Omega'=Z^{\Zz}$ endowed with the product $\sigma$-algebra $\Sigma' := \ZZ^{\Zz}$. For each $n \in \Zz$, let $\pi_n \colon \Omega' \to Z$ be the projection given by $\pi_n(\omega')=\omega'_n$ for all $\omega'=(\omega'_n)_{n \in \Zz} \in \Omega'$. Since  $Z$ is a separable complete metric space, there exists a probability measure $\Pp_{\zeta}$ on $\Sigma'$ such that the coordinate process $\{\pi_n\}_{n \in \Zz}$ is a Markov chain on $(\Omega',\Sigma',\Pp_{\zeta})$ with transition probability $P_z$ and stationary probability measure $\zeta$. Moreover, the shift map $\sigma' \colon \Omega' \to \Omega'$ 
is measurable with respect to $\Sigma'$ and preserves $\Pp_{\zeta}$  (see \cite[Theorem 3.1.7 and Example 5.1.3]{Douc:2018wc}).

Let $A \colon Z \to \operatorname{GL}(d,\Rr)$ be a measurable map such $\log^+ \|A(\cdot)^{\pm 1}\|$ belong to $L^1(\zeta)$. Consider the linear cocycle over $(\Omega',\Sigma',\Pp_{\zeta},\sigma')$ generated by $A$, and denote by $\Lambda^+_A(\omega')$ and $\Lambda^-_A(\omega')$ its maximal and minimal Lyapunov exponents, respectively \cite{Ledrappier:1986ue}. 

By abuse of notation, we also denote by $A(z)$ the action induced by the matrix $A(z)$ on the projective space $\Pp^{d-1}$.

A family of probability measures $\{\eta_z : z \in Z\}$ on $\mathbb{P}^{d-1}$ is said to be \emph{measurable} if for every Borel set $B \subset \mathbb{P}^{d-1}$, the map
$z \mapsto \eta_z(B)$ is $\mathcal Z$-measurable. Since $\mathbb{P}^{d-1}$ is a compact metrizable space, this condition is equivalent to requiring that the map
$z \mapsto \eta_z$ is measurable as a map from $(Z,\mathcal Z)$ into the space of probability measures on $\mathbb{P}^{d-1}$ endowed with the Borel $\sigma$-algebra of the weak-* topology.

\begin{theorem}[{\cite[Corollary 2]{Ledrappier:1986ue}}]
\label{thm: Ledrappier}
Suppose that 
\[
\Lambda^+_A(\omega') = \Lambda^-_A(\omega') \quad \text{for } \Pp_{\zeta} \text{-a.e. } \omega' \in \Omega'.
\] 
Then there exists a measurable family $\{\eta_z : z \in Z\}$ of probability measures on $\mathbb{P}^{d-1}$ such that for $\zeta$-a.e. $z \in Z$, 
\[
A(z)_{*} \eta_z = \eta_{z_1} \quad \text{for } P_z  \text{-a.e. } z_1 \in Z.
\]
\end{theorem}



\subsubsection{Invariant principle for random maps}

We now establish a result (Theorem~\ref{thm:G}) that may be regarded as a generalization of Furstenberg’s theorem to random maps satisfying Assumptions~\ref{as:one} and~\ref{as:three}. Related results for random diffeomorphisms on manifolds were previously obtained by Carverhill~\cite{carverhill1987furstenberg} and Baxendale~\cite{Baxendale:1989aa}.

Denote by $PM = \bigsqcup_{y \in M} P_yM$ the projective bundle of the manifold $M$, where $P_yM$ is the projective space of $T_xM$. Let $\pi \colon PM \to M$ be the bundle projection. 

Any probability measure $\eta$ on $PM$ such that $\pi_* \eta = m$ admits a \emph{disintegration} with respect to $m$ (see \cite{Viana:2014uo}), i.e. a measurable family $\{\eta_y \colon y \in M\}$ of probability measures on $PM$ such that each $\eta_y$ is concentrated on the fiber $P_yM $ and   
\[
\eta(E) = \int_M \eta_y(E \cap P_y M) \, dm(y) \quad \text{for every measurable } E \subset PM.
\]  

We use the same notation $Df_x(y)$ for the differential of $f_x$ at $y$ and for the map it induces on the projective fiber $P_yM$.


\begin{theorem}
\label{thm:G}
Let $F \colon \Omega \times M \to \Omega \times M$ be a random dynamical system generated by $f$ and $\nu$ that satisfies Assumptions~\ref{as:one} and \ref{as:three}. Suppose that
\[
\lambda^+_F(\omega,y)=\lambda^-_F(\omega,y) \quad \text{for } \rho_\nu \times m \text{-a.e. } (\omega,y) \in \Omega \times M.
\] 
Then there exists a probability measure $\eta$ on $PM$ with $\pi_* \eta = m$ such that 
\[
(Df_x)_* \eta = \eta \quad \text{for } \nu \text{-a.e. } x \in X.
\]
In particular, if $\{\eta_y : y \in M\}$ denotes the disintegration of $\eta$ with respect to $m$, then 
\[
Df_x
(y)_* \eta_y = \eta_{f_x(y)} \quad \text{for } (\nu \times m) \text{-a.e. } (x,y) \in X \times M.
\]
\end{theorem}

\begin{proof}
The argument below follows closely the proof of \cite[Theorem 1]{carverhill1987furstenberg}.

Consider the Markov chain $\{(X_n,Y_n) \colon n \in \Nn\}$ on $X \times M$ and transition probabilities given by
\[
P_{(x,y)}(B \times C) = 	\nu(B) \delta_{f_x(y)}(C), \quad  B \in \XX, C \in \BB.
\]
for all $x \in X, y \in M$.
This means that $\{X_n\}$ is a sequence of i.i.d. random variables with values in $X$ and distribution $\nu$, whereas $\{Y_n\}$ is a Markov chain with state space $M$ satisfying the recursive relation   
\[
Y_{n+1}=f_{X_n}(Y_n).
\]
It is straightforward to check that $\mu:=\nu \times m$ is a stationary probability for $\{(X_n,Y_n)\}$.

Let $A\colon\Omega\times M\to\operatorname{GL}(d,\Rr)$ be the extension, defined after Assumption~\ref{as:one}, of the matrix representation of the derivative cocycle. Since $A(\omega,y)$ depends on $\omega$ only through its zeroth coordinate $\omega_0$, the map
\[
A' \colon X\times M\to\operatorname{GL}(d,\Rr),
\qquad
A'(x,y)=A(\omega,y),
\]
where $\omega\in\Omega$ is any sequence with $\omega_0=x$ is well defined.

Under the hypotheses of the theorem, we may apply Theorem~\ref{thm: Ledrappier}
to the Markov chain $\{(X_n,Y_n)\}$ and the matrix function $A'$. This yields a measurable family of probability measures on $\mathbb{P}^{d-1}$,
\[
\{\bar{\eta}_{(x,y)} : (x,y) \in X \times M\},
\]
and a measurable set $U \subset X \times M$ with $(\nu \times m)(U)=1$ such that for every $(x,y) \in U$,
\[
\bar{\eta}_{(x_1,y_1)}
=
Df_x(y)_{*}\,\bar{\eta}_{(x,y)}
\quad
\text{for $P_{(x,y)}$-a.e.\ $(x_1,y_1) \in X \times M$.}
\]
Using the definition of $P_{(x,y)}$, this condition can be rewritten as follows:
for every $(x,y) \in U$,
\begin{equation}
\label{eqn: 1}
\bar{\eta}_{(x_1,f_x(y))}
=
Df_x(y)_{*}\,\bar{\eta}_{(x,y)}
\quad
\text{for $\nu$-a.e.\ $x_1 \in X$.}
\end{equation}
It follows that for each $(x,y) \in U$, the measure $\bar{\eta}_{(x_1,f_x(y))}$ is independent of $x_1$ for $\nu$-almost every $x_1 \in X$.

We now introduce another measurable family $\{\eta_y \colon y \in M\}$ of probability measures on $\Pp^{d-1}$ defined by 
\[
\eta_y(B) = \int_X \bar{\eta}_{(x,y)}(B) \,d \nu(x) 
\]
for every  $y \in M$ and every measurable set $B \subset \Pp^{d-1}$.  Then Property~\eqref{eqn: 1} implies that $Df_x(y)_* \bar{\eta}_{(x,y)} = \eta_{f_x(y)}$ for every $(x,y) \in U$.

Let 
\[
V = \left\{ (x,y) \in X \times M : \bar{\eta}_{(x,y)} = \eta_y \right\}.
\]
This set is measurable, since $(x,y) \mapsto \bar{\eta}_{(x,y)}$ and $(x,y) \mapsto \eta_y$ are measurable maps. 

Next, we show that $(\nu \times m)(V)=1$. For every $x \in X$, let
\[
U_x = \left \{ y \in M : (x,y) \in U\right \}
\]
be the $x$-section of $U$. Since $(\nu \times m)(U)=1$,  Fubini's Theorem implies that there exists $\bar{x} \in X$ such $m(U_{\bar{x}})=1$. By hypothesis, the map $f_{\bar{x}} \colon N_{\bar{x}} \to M$ is an $m$-preserving $C^1$ embedding and $m(N_x)=1$. Therefore, the set $M_{\bar{x}}:= f_{\bar{x}}(U_{\bar{x}} \cap N_{\bar{x}})$ is measurable and $m(M_{\bar{x}})=1$. By the definition of $U_{\bar{x}}$ and Property~\eqref{eqn: 1}, if $y \in M_{\bar{x}}$, then $(x,y) \in V$ for $\nu$-a.e. $x \in X$. Equivalently, 
\[
\nu(V^y)=1 \quad \text{for every } y \in M_{\bar{x}},
\]
where $V^y = \{x \in X : (x,y) \in V\}$ is the $y$-section of $V$. Finally, by Fubini's Theorem,  we obtain
\[
(\nu \times m)(V) = \int_M \nu(V^y) dm(y) \ge  \int_{M_{\bar{x}}} \nu(V^y) dm(y)  =1.
\]

Let $W = U \cap V$. Then $W$ is measurable and satisfies $(\nu \times m)(W)=1$. Take $(x,y) \in W$. Since $(x,y) \in U$, we have $Df_x(y)_* \bar{\eta}_{(x,y)}= \eta_{f_x(y)}$. On the other hand, $(x,y) \in V$ implies that $\bar{\eta}_{(x,y)} = \eta_y$. Combining these identities, we obtain $Df_x(y)_* \eta_y= \eta_{f_x(y)}$.

To complete the proof, consider the probability measure $\eta$ on $PM$ whose disintegration over $m$ is given by $\{\eta_y : y \in M\}$. By the conclusion above, $\eta$ satisfies $(Df_x)_* \eta = \eta$ for $\nu$-a.e. $x \in X$.
%
%
\end{proof}

\subsubsection{Random additive perturbations of toral maps}
\label{Consequences}

By Lemma~\ref{le:add-pert}, Theorem~\ref{thm:G} applies to every $\nu$-random additive perturbation $F$ of a toral map $g$. Moreover, for such a system, the condition $\lambda_F^+=\lambda_F^-$ almost everywhere is equivalent to $\lambda_F^+=0$ almost everywhere. Since the projective tangent bundle of $\Tt^d$ is trivial, $P\Tt^d\cong \Tt^d\times\Pp^{d-1}$,
Theorem~\ref{thm:G} immediately yields the following corollary.

\begin{corollary}
\label{cor:H}
Let $F$ be $\nu$-random additive perturbation of a toral map $g$. Suppose that 
\[
\lambda^+_F(\omega,y)=0 \quad \text{for } (\rho_\nu \times m) \text{-a.e. } (\omega,y) \in \Omega \times \Tt^d.
\]
 Then there exists a measurable family of probability measures $\{\eta_y : y \in \Tt^d\}$ on $\Pp^{d-1}$ such that for $m$-a.e. $y \in M$, 
\[
Dg(y)_* \eta_y = \eta_{g(y)+x} \quad \text{for } \nu \text{-a.e. } x \in \Tt^d.
\]
\end{corollary}


We now examine some consequences of Corollary~\ref{cor:H} under the additional assumption that $\nu$ is absolutely continuous with respect to
$\operatorname{vol}$.

Recall that $\SL^{\pm}(d,\Rr)$ denotes the set of all real $d\times d$ matrices with determinant $\pm1$. Also denote by $\SL(d,\Rr)$ the set of those with determinant
$1$. Let $\mu$ be the push-forward of $m|_N$ under the map $Dg\colon N\to \SL^{\pm}(d,\Rr)$. Let $H\subset \SL^{\pm}(d,\Rr)$ denote the support of $\mu$. For $h\in H$, we use the same symbol to denote both the matrix $h$ and the induced projective transformation of $\Pp^{d-1}$.


\begin{theorem}
\label{thm:invariant}
Let $F$ be a $\nu$-random additive perturbation of a toral map $g$.
Suppose that $\nu\ll m$ and that $\lambda_F^+(\omega,y)=0$ for $(\rho_\nu\times m)$-a.e. $(\omega,y)\in\Omega\times\Tt^d$. Then there exists a probability measure $\eta$ on $\mathbb P^{d-1}$ such that
\[
h_*\eta=\eta
\quad\text{for all }h\in H.
\]
Moreover, one of the following alternatives holds:
\begin{enumerate}
\item $H$ is contained in a compact subgroup of $\SL^{\pm}(d,\Rr)$,
\item there exists a nonempty finite family $L$ of proper subspaces of $\Rr^d$
such that $h(L)=L$ for every $h\in H$.
\end{enumerate}
\end{theorem}

\begin{proof}
Let $p=d\nu/dm \in L^1(m)$, and set $E=\{x\in\Tt^d:p(x)>0\}$. We have $m(E)>0$. By Corollary~\ref{cor:H}, there exists a measurable
family of probability measures $\{\eta_y:y\in\Tt^d\}$ on $\mathbb P^{d-1}$ and a full $m$-measure set $U\subset N$ such that for every $y\in U$,
\[
Dg(y)_*\eta_y=\eta_{g(y)+x}
\quad\text{for }m\text{-a.e. }x\in E.
\]

Let $V=g(U)$. Since $g|_N$ is a $C^1$ embedding preserving $m$ and $U$ has full $m$-measure, the set $V$ has full $m$-measure. Therefore, for every
$v = g(y) \in V$, we have
\[
Dg(y)_*\eta_y = \eta_{v+x} \quad\text{for }m\text{-a.e. }x\in E.
\]
This implies that the map $w\mapsto\eta_w$ is constant $m$-a.e. on $v+E$. Denote this constant measure by $\zeta_v$.

We now prove that these measures $\zeta_v$ are equal for all $v \in V$. Define
\[
D=\{t\in\Tt^d:m(E\cap(E+t))>0\}.
\]
Since $m(E)>0$, Steinhaus' Theorem~\cite[Corollary 20.17]{hewittross} implies that $D$ contains a neighborhood
$W$ of $0$. If $v,v'\in V$ and $v'-v\in W$, then
\[
m\bigl((v+E)\cap(v'+E)\bigr)
=
m\bigl(E\cap(E+v'-v)\bigr)>0.
\]
Since $w\mapsto\eta_w$ is constant $m$-a.e. on each of the two translates
$v+E$ and $v'+E$, the corresponding constants must coincide.

Choose a symmetric neighborhood $W_0$ of $0$ such that $W_0\subset W$. Since $\Tt^d$ is connected and compact, there exists $N_0\ge 1$ such that
every element of $\Tt^d$ can be written as a sum of $N_0$ elements of $W_0$.
Let $v,v'\in V$, and define $\mathcal C(v,v')$ as the set of tuples
\[
(v_1,\ldots,v_{N_0-1})\in(\Tt^d)^{N_0-1}
\]
such that
\[
v_1-v\in W_0,\quad v_2-v_1\in W_0,\quad \ldots,\quad v'-v_{N_0-1}\in W_0.
\]
Then $\mathcal C(v,v')$ is a nonempty open subset of $(\Tt^d)^{N_0-1}$ and intersects $V^{N_0-1}$, since $V$ has full $m$-measure. Hence any two points $v,v'\in V$ can
be joined by a finite chain
\[
v=v_0,v_1,\ldots,v_{N_0}=v',
\qquad v_i\in V,
\]
with
\[
v_{i+1}-v_i\in W_0
\quad\text{for every }i=0,\ldots,N_0-1.
\]
It follows that $\zeta_v = \zeta_{v'}$ for all $v,v' \in V$. Hence there exists a probability measure $\eta$ on $\mathbb P^{d-1}$ such that for every $v\in V$,
\[
\eta_w=\eta \quad\text{for }m\text{-a.e. }w\in v+E.
\]

Let $A=\{w\in\Tt^d:\eta_w\neq\eta\}$. It is straightforward to see that $A$ is measurable. By the previous conclusion, 
\[
m(A\cap(v+E))=0 \quad \text{for every } v\in V.
\]
By Fubini's Theorem,
\[
0 = \int_V m(A\cap(v+E))\,dm(v) = \int_A m\bigl(V\cap(w-E)\bigr)\,dm(w).
\]
Since $V$ has full $m$-measure, $m(V\cap(w-E))=m(E)>0$ for every $w\in\Tt^d$. Therefore $m(A)=0$, and so
\[
\eta_w=\eta
\quad\text{for }m\text{-a.e. }w\in\Tt^d.
\]

We can then conclude that
\[
Dg(y)_*\eta=\eta
\quad\text{for }m\text{-a.e. }y\in\Tt^d,
\]
or equivalently
\[
A_*\eta=\eta
\quad\text{for }\mu\text{-a.e. }A\in\SL^{\pm}(d,\Rr).
\]
Since the stabilizer of $\eta$ is closed, this identity holds for every
$A\in\supp\mu=H$. Hence
\[
h_*\eta=\eta
\quad\text{for every }h\in H.
\]

The final alternative follows by applying~\cite[Corollary~3.2.2]{zimmer} to the $H$-invariant probability measure $\eta$ on $\mathbb P^{d-1}$ (see also~\cite[Section~7.3]{Viana:2014uo}). Therefore either
$H$ is contained in a compact subgroup of $\SL^{\pm}(d,\Rr)$, or there exists a
nonempty finite family $L$ of proper subspaces of $\Rr^d$ such that
\[
h(L)=L
\quad\text{for every }h\in H.
\]
This proves the proposition.
\end{proof}

\begin{remark} \label{remark:hyp-ell}
The failure of Conditions (1) and (2) in the conclusion of Theorem~\ref{thm:invariant} implies that $\lambda^+_F > 0$ on a set of positive $(\rho_\nu \times m)$-measure. This provides a criterion for the positivity of the maximal Lyapunov exponent of $F$ analogous to Furstenberg's criterion for random matrices~\cite[Section~7.3]{Viana:2014uo}.
\end{remark}

When $d=2$, the last part of Theorem~\ref{thm:invariant} reduces to the following characterization (see~\cite[Section 6.3]{Viana:2014uo}):  
\begin{enumerate}
    \item $H$ is contained in a compact subgroup of $\SL^{\pm}(2,\mathbb{R})$, or  
    \item there exists a nonempty set $L\subset\mathbb{R}^2$ consisting of one or two $1$-dimensional subspaces such that $h(L)=L$ for every $h\in H$.
\end{enumerate}  
The second case can be further refined into two complementary possibilities:  
\begin{itemize}
    \item $L$ consists of either one or two invariant $1$-dimensional subspaces,
    \item $L$ consists of two $1$-dimensional subspaces interchanged by some $h \in H$.
\end{itemize}


The following corollary gives a simple condition ensuring that $\lambda_F^{+}>0$ on a set of positive measure when $d=2$. It will be used in Section~\ref{sec:random_standard} to establish the positivity of $\lambda_F^{+}$ when $F$ is a random additive perturbation of a generalized standard map.

\begin{corollary}
\label{cor:simple}
Assume that the hypotheses of Theorem~\ref{thm:invariant} hold and that $d=2$. Suppose there exist elements $h_1, h_2 \in H \cap \SL(2,\Rr)$ such that $h_1 \neq \pm I$, $|\tr h_1| \ge 2$ and $0 < |\tr h_2| < 2$. 
Then $\lambda^+_F(\omega,y)>0$ on a set of positive $(\rho_\nu \times m)$-measure.
\end{corollary}

\begin{proof}
Since $h_1\neq \pm I$ and $|\operatorname{tr} h_1|\geq 2$, the matrix $h_1$
is either hyperbolic, or parabolic and not diagonalizable. Hence $h_1$ cannot belong to a compact subgroup of $\SL^{\pm}(2,\Rr)$, so Condition~\textup{(1)} in Theorem~\ref{thm:invariant} is violated.

Moreover, since $0<|\operatorname{tr} h_2|<2$, the matrix $h_2$ is elliptic and is not conjugate to a rotation by angle $0$, $\pi/2$, $\pi$ or $3\pi/2$. Therefore the projective action of $h_2$ has neither fixed points nor invariant two-point sets, and Condition~\textup{(2)} is also violated.

Thus neither condition in Theorem~\ref{thm:invariant} holds, and the conclusion follows from that theorem.
\end{proof}




The following corollary will be applied in Section~\ref{sec:billiards}, where we study billiard maps and encounter a sequence of matrices in $H$
with the following degenerating factorization.

\begin{corollary}
\label{cor:K}
Assume that the hypotheses of Theorem~\ref{thm:invariant} hold
and that $d=2$. Suppose that there exists a sequence $\{h_n\}_{n\geq 0}\subset H$ such that
\[
h_n=A_nD_nB_n,
\]
where
\[
A_n=
\begin{pmatrix}
1 & 0\\
0 & \alpha_n
\end{pmatrix},
\qquad
\alpha_n \xrightarrow[n\to\infty]{}0,
\]
\[
D_n=
\begin{pmatrix}
1+a_n & b+b_n\\
c_n & 1+d_n
\end{pmatrix},
\qquad
b\neq0,
\qquad
a_n,b_n,c_n,d_n \xrightarrow[n\to\infty]{}0,
\]
and
\[
B_n=\beta_n
\begin{pmatrix}
\beta_n^{-1} & 0\\
0 & 1
\end{pmatrix},
\qquad
\beta_n \xrightarrow[n\to\infty]{}+\infty.
\]
Then
\[
\eta=\delta_{\hat e},
\]
where $\hat e=[1:0]\in\Pp^1$.
\end{corollary}

\begin{proof}
Let $\widehat{h}_n \colon \mathbb P^1 \to \mathbb P^1$ denote the map induced by the action of $h_n$ on $\mathbb P^1$. It is enough to prove that
\[
(\widehat{h}_n)_*\eta \xrightarrow[n\to\infty]{\mathrm{w}*}\delta_{\hat e},
\]
Indeed, since $h_n\in H$ and $\eta$ is $H$-invariant, it follows that $\eta = \delta_{\hat e}$.

Let $[u:v]\in\Pp^1$. The projective action of $h_n=A_nD_nB_n$ is given by
\[
\widehat h_n([u:v])
=
\left[
\beta_n^{-1}(1+a_n)u+(b+b_n)v:
\alpha_n\beta_n^{-1}c_nu+\alpha_n(1+d_n)v
\right].
\]
If $v\neq0$, then 
\[
\widehat h_n([u:v])\longrightarrow [bv:0]=\hat e.
\]
If $v=0$, then $u\neq0$, and after canceling the common factor
$\beta_n^{-1}u$, we obtain
\[
\widehat h_n([u:0])
=
[1+a_n:\alpha_n c_n]
\longrightarrow \hat e.
\]
Therefore $\widehat h_n(\xi)\to\hat e$ for every $\xi\in\Pp^1$. 

By the Dominated Converge Theorem, for every continuous function $\varphi\colon\Pp^1\to\Rr$,
\[
\int_{\Pp^1}\varphi(\xi)\,d(\widehat h_{n*} \eta)(\xi)
=
\int_{\Pp^1}\varphi(\widehat h_n(\xi))\,d\eta(\xi)
\longrightarrow
\varphi(\hat e).
\]
Hence $\widehat h_{n*} \eta\to\delta_{\hat e}$ in the weak-$*$ topology. 
\end{proof}

%
%

\section{Ergodicity of random additive perturbations of toral maps}
\label{se:ergodicity}

In this section, we study stationary measures for the Markov chains generated by random additive perturbations of toral maps; see Definition~\ref{def:toral_map} for the precise meaning of toral map.

Let $g\colon\Tt^d\to\Tt^d$ be a toral map and let $\nu$ be a probability measure on $\Tt^d$. For each $x\in\Tt^d$, define
\[
f_x\colon\Tt^d\to\Tt^d,
\qquad
f_x(y)=g(y)+x.
\]
Let $(X_n){n\geq 0}$ be an i.i.d.\ sequence of $\Tt^d$-valued random variables with common distribution $\nu$. We consider the Markov chain $(Y_n){n\geq 0}$ on $\Tt^d$ defined by
\[
Y_{n+1}=f_{X_n}(Y_n)=g(Y_n)+X_n.
\]
The corresponding transition probabilities are
\[
P_y(E)=\nu\bigl(\{x\in\Tt^d:g(y)+x\in E\}\bigr),
\qquad E\subset \Tt^d \text{ measurable}.
\]
The associated Markov operator is
\[
P\varphi(y)=\int_{\Tt^d}\varphi(g(y)+x)\,d\nu(x)
\]
for every bounded Borel function $\varphi\colon\Tt^d\to\Rr$.

We denote by $P^*$ the dual operator acting on probability measures on $\Tt^d$. A probability measure $\mu$ is called $P$-stationary if $P^*\mu=\mu$. The Markov operator $P$ is called \emph{uniquely ergodic} if it admits a unique stationary probability measure. It is called \emph{strong Feller} if it maps
bounded Borel functions to continuous functions.

We impose additional conditions on the map $g$ and the measure $\nu$. More precisely, we consider the following two cases.

\medskip

\noindent\textbf{Case (i).}
\begin{itemize}
\item $g \colon \Tt^d \to \Tt^d$ is continuous and preserves $m$,
\item $\nu$ is absolutely continuous with respect to $m$.
\end{itemize}

\medskip

\noindent\textbf{Case (ii).}
The natural numebr $d$ is even and write $\Tt^d=\Tt^{d/2}\times\Tt^{d/2}$. Moreover,
\begin{itemize}
\item $g=(g_1,g_2)\colon\Tt^d\to\Tt^{d/2}\times\Tt^{d/2}$ is $C^1$ and
preserves $m$,
\item $\nu=\delta_0\times \chi$, where $\delta_0$ is the Dirac measure at $0\in\Tt^{d/2}$, and $\chi=h\,m_1$; here $m_1$ denotes the normalized Lebesgue measure on the second factor $\Tt^{d/2}$, $h\in L^1(m_1)$, $h\geq 0$, and
\[
\int_{\Tt^{d/2}}h(s)\,dm_1(s)=1,
\]
\item for every $y\in\Tt^d$, the map
\[
H_y\colon\Tt^{d/2}\to\Tt^{d/2},
\qquad
H_y(s)=g_1\bigl(g(y)+(0,s)\bigr),
\]
is a $C^1$ diffeomorphism.
\end{itemize}

These two cases cover the random additive perturbations of billiard maps and generalized standard maps considered later in the paper.

\begin{lemma}
\label{lem:strong-feller-ac}
Suppose that Case (i) holds. Then $P$ is strong Feller.
\end{lemma}

\begin{proof}
Since $\nu\ll m$, there exists $p\in L^1(m)$, $p\geq 0$, such that
\[
d\nu(z)=p(z)\,dm(z).
\]
For $w\in\Tt^d$, let $U_w\colon L^1(m)\to L^1(m)$ be the translation operator
$U_w\psi(z)=\psi(z-w)$. The family $\{U_w:w\in\Tt^d\}$ is strongly continuous on $L^1(m)$.

Let $\varphi\colon\Tt^d\to\Rr$ be bounded and Borel. Using the translation-invariance of $m$, we obtain
\[
\begin{aligned}
P\varphi(y)
&=
\int_{\Tt^d}\varphi(g(y)+z)p(z)\,dm(z) =
\int_{\Tt^d}\varphi(z)p(z-g(y))\,dm(z)  \\
&=
\int_{\Tt^d}\varphi(z)U_{g(y)}p(z)\,dm(z).
\end{aligned}
\]
Therefore, for $y_1,y_2\in\Tt^d$,
\[
\left|P\varphi(y_1)-P\varphi(y_2)\right|
\leq
\|\varphi\|_\infty
\left\|U_{g(y_1)}p-U_{g(y_2)}p\right\|_{L^1}.
\]
Since $g$ is continuous and $w\mapsto U_w p$ is continuous as a map from $\Tt^d$ to $L^1(m)$, the last term converges to $0$ as $y_2\to y_1$. Hence $P\varphi$ is continuous. Thus $P$ is strong Feller.
\end{proof}

\begin{lemma}
\label{lem:strong-feller-singular}
Suppose that Case (ii) holds. Then $Q:=P^2$ is strong Feller.
\end{lemma}

\begin{proof}
Let $\varphi\colon\Tt^d\to\Rr$ be bounded and Borel. Since $\nu=\delta_0\times h\,m_1$, we may write the noise variables as $(0,s)$, with $s\in\Tt^{d/2}$. For $y\in\Tt^d$, define
\[
V_y(s)=g_2\bigl(g(y)+(0,s)\bigr),
\qquad s\in\Tt^{d/2}.
\]
Then
\[
f_{(0,t)}\circ f_{(0,s)}(y)
=
\bigl(H_y(s),V_y(s)+t\bigr),
\]
and therefore
\[
Q\varphi(y)
=
\int_{\Tt^{d/2}}\int_{\Tt^{d/2}}
\varphi\bigl(H_y(s),V_y(s)+t\bigr)h(s)h(t)\,dm_1(s)\,dm_1(t).
\]

Fix $y\in\Tt^d$, and let $\tau_y=H_y^{-1}$. Define
\[
\Phi_y\colon\Tt^{d/2}\times\Tt^{d/2}\to\Tt^d, \qquad \Phi_y(s,t)=\bigl(H_y(s),V_y(s)+t\bigr).
\]
Since $H_y$ is a $C^1$ diffeomorphism, $\Phi_y$ is a $C^1$ diffeomorphism of $\Tt^d$. Its inverse is
\[
\Phi_y^{-1}(z_1,z_2)
=
\bigl(\tau_y(z_1),\,z_2-V_y(\tau_y(z_1))\bigr).
\]
Moreover,
\[
\left|\det D\Phi_y(s,t)\right|
=
\left|\det DH_y(s)\right|.
\]
Thus, by the change of variables $z=\Phi_y(s,t)$, we can write
\[
Q\varphi(y)
=
\int_{\Tt^d}\varphi(z)\alpha_y(z)\,dm(z),
\]
where, for $z=(z_1,z_2)\in\Tt^{d/2}\times\Tt^{d/2}$,
\[
\alpha_y(z_1,z_2)
=
\frac{
h(\tau_y(z_1))
h\bigl(z_2-V_y(\tau_y(z_1))\bigr)
}{
\left|\det DH_y(\tau_y(z_1))\right|
}.
\]

We first prove that $y\mapsto\alpha_y$ is continuous as a map from $\Tt^d$ to $L^1(\Tt^d)$ when $h$ is continuous. Consider
\[
G\colon\Tt^d\times\Tt^{d/2}\to\Tt^d\times\Tt^{d/2},
\qquad
G(y,s)=(y,H_y(s)).
\]
The map $G$ is continuous and bijective. Since the domain is compact and the codomain is Hausdorff, $G^{-1}$ is continuous. Hence
\[
(y,z_1)\mapsto \tau_y(z_1)
\]
is continuous. Moreover, $(y,s)\mapsto V_y(s)$ and $(y,s)\mapsto DH_y(s)$ are continuous. Since $\det DH_y(s)\neq 0$ for every $(y,s)$, compactness implies that
$(y,s)\mapsto \left|\det DH_y(s)\right|^{-1}$ is continuous and bounded. Therefore, if $h$ is continuous, the map $(y,z)\mapsto \alpha_y(z)$ is continuous on $\Tt^d\times\Tt^d$. In particular,
\[
\|\alpha_y-\alpha_{y_0}\|_{L^1}\longrightarrow 0
\qquad\text{as }y\to y_0.
\]

We now pass to the general case $h\in L^1(m_1)$. Choose $h_n\in C(\Tt^{d/2})$ such that $h_n\to h$ in $L^1(m_1)$. Let $\alpha_y^{(n)}$ be the function obtained from the previous formula by replacing $h$ with $h_n$. Equivalently, if $L_y$ denotes the Perron--Frobenius operator associated with the diffeomorphism $\Phi_y$, then
\[
\alpha_y=L_y(h\otimes h),
\qquad
\alpha_y^{(n)}=L_y(h_n\otimes h_n).
\]
Since $L_y$ is an isometry on $L^1(\Tt^d)$, we have
\[
\begin{aligned}
\|\alpha_y-\alpha_y^{(n)}\|_{L^1}
&=
\|L_y(h\otimes h-h_n\otimes h_n)\|_{L^1} \\
&=
\|h\otimes h-h_n\otimes h_n\|_{L^1} \\
&\leq
\|h-h_n\|_{L^1(m_1)}\|h\|_{L^1(m_1)}
+
\|h_n\|_{L^1(m_1)}\|h-h_n\|_{L^1(m_1)}.
\end{aligned}
\]
The right-hand side is independent of $y$ and converges to $0$ as $n\to\infty$. Since $y\mapsto\alpha_y^{(n)}$ is continuous from $\Tt^d$ to $L^1(\Tt^d)$ for each $n$, it follows that
\[
\|\alpha_y-\alpha_{y_0}\|_{L^1}\longrightarrow 0
\qquad\text{as }y\to y_0
\]
also for $h\in L^1(m_1)$.

Consequently,
\[
\left|Q\varphi(y)-Q\varphi(y_0)\right|
\leq
\|\varphi\|_\infty\|\alpha_y-\alpha_{y_0}\|_{L^1}
\longrightarrow 0
\qquad\text{as }y\to y_0.
\]
Thus $Q=P^2$ maps bounded Borel functions to continuous functions. Hence $Q$ is strong Feller.
\end{proof}

\begin{proposition}
\label{prop:unique-ergodicity-additive}
Assume that either Case (i) or Case (ii) holds. Then $m$ is the unique $P$-stationary probability measure. In particular, $m$ is ergodic for $P$.
\end{proposition}

\begin{proof}
First observe that $m$ is $P$-stationary. Indeed, for every bounded Borel function $\varphi\colon\Tt^d\to\Rr$,
\[
\int_{\Tt^d}P\varphi(y)\,dm(y)
=
\int_{\Tt^d}\int_{\Tt^d}\varphi(g(y)+x)\,d\nu(x)\,dm(y).
\]
For each fixed $x\in\Tt^d$, the map $y\mapsto g(y)+x$ preserves $m$, because both $g$ and translations preserve $m$. Hence $P^*m=m$.

Assume first that Case (i) holds. By Lemma~\ref{lem:strong-feller-ac}, $P$ is strong Feller. In addition, $\Tt^d$ is connected, $m$ is $P$-stationary and $\operatorname{supp}m=\Tt^d$. Hence, it follows from~\cite[Part (iii) of Proposition~5.18]{Benaim_2022} that $m$ is the unique $P$-stationary probability measure.

Assume now that Case (ii) holds. By Lemma~\ref{lem:strong-feller-singular}, the operator $Q=P^2$ is strong Feller. Since $m$ is $P$-stationary, it is also $Q$-stationary. Again, $\operatorname{supp}m=\Tt^d$, so~\cite[Part (iii) of Proposition~5.18]{Benaim_2022}, applied to $Q$, implies that $m$ is the unique $Q$-stationary probability measure.

Let $\mu$ be any $P$-stationary probability measure. Then $\mu$ is $Q$-stationary.  By uniqueness of the $Q$-stationary probability measure, $\mu=m$. Therefore $m$ is the unique $P$-stationary probability measure.

Finally, for both cases, uniqueness of the stationary probability measure implies that $m$ is ergodic for $P$.
\end{proof}

\begin{corollary}
\label{cor:skew-product-ergodicity}
Assume that either Case (i) or Case (ii) holds. Then the invariant probability measure $\rho_\nu\times m$ is ergodic for the skew-product $F$.
\end{corollary}

\begin{proof}
By Proposition~\ref{prop:unique-ergodicity-additive}, the measure $m$ is ergodic for the Markov operator $P$. Therefore, by~\cite[Proposition~5.13]{Viana:2014uo}, the invariant probability measure $\rho_\nu\times m$ is ergodic for $F$.
\end{proof}

\begin{remark}
\label{rem:ergodic}
Since $(\rho_\nu\times m)$ is ergodic by Corollary~\ref{cor:skew-product-ergodicity}, the Lyapunov exponents $\lambda_F^-(\omega,y)$ and $\lambda_F^+(\omega,y)$ are constant for $(\rho_\nu\times m)$-almost every $(\omega,y)\in\Omega\times\Tt^d$.
\end{remark}

The unique ergodicity of the Markov operator also gives the following equidistribution result for random orbits.

\begin{proposition}
\label{pr:equidistribution}
Assume that either Case (i) or Case (ii) holds. Then for every $y\in\Tt^d$ and for $\rho_\nu$-almost every $\omega\in\Omega$, the random orbit $\{F_\omega^n(y)\}_{n\geq 0}$ is equidistributed with respect to $m$:
\[
\frac{1}{n}\sum_{k=0}^{n-1}\delta_{F_\omega^k(y)}
\;\xrightarrow[n\to\infty]{\mathrm{w}^*}\;
m \qquad \text{for } \rho_\nu \text{-a.e.} \omega\in\Omega.
\]
In particular, $\rho_\nu$-a.e. random orbit of $y$ is dense in $\Tt^d$.
\end{proposition}

\begin{proof}
By Proposition~\ref{prop:unique-ergodicity-additive}, $P$ admits the unique stationary probability measure $m$. Therefore, Breiman's law of large numbers for Markov chains~\cite[Corollary~2.7]{BenoistQuint}, applied to the uniquely ergodic Markov kernel $P$, gives
\[
\frac{1}{n}\sum_{k=0}^{n-1}\delta_{F_\omega^k(y)}
\;\xrightarrow[n\to\infty]{\mathrm{w}^*}\;
m
\]
for $\rho_\nu$-a.e. $\omega\in\Omega$.

Finally, since $m$ assigns positive measure to every nonempty open subset of $\Tt^d$, equidistribution implies that $\rho_\nu$-a.e. random orbit of $y$ is dense in $\Tt^d$.
\end{proof}

\section{Random billiards}
\label{sec:billiards}
\subsection{Convex billiards on constant-curvature surfaces}

In this section, we recall the fundamental definitions and properties of billiards in convex domains on surfaces of constant curvature. For a comprehensive treatment of billiards on general surfaces, we refer to~\cite{DiasCarneiro:2024}, and for the case of surfaces with constant curvature to~\cite{dosSantos:2017}.

Let $S$ denote one of the standard surfaces of constant curvature $K$: the Euclidean plane $\mathbb{E}^2$ ($K=0$), the sphere $\mathbb{S}^2$ ($K=1$) and the hyperbolic plane $\mathbb{H}^2$ ($K=-1$). Let $D \subset S$ be a domain whose boundary $\partial D$ is a $C^2$ simple closed convex curve with positive geodesic curvature (such a curve is called an \emph{oval} in~\cite{dosSantos:2017,DiasCarneiro:2024}). We parametrize $\partial D$ by arc-length $s$, normalized so that the total length of the curve is 1, i.e. $|\partial D|=1$.

A \emph{billiard} with table $D$ is the mechanical system consisting of a point particle that moves inside $D$ along geodesics of $S$ and reflects elastically upon colliding with the boundary $\partial D$, obeying the usual law of reflection: the angles of reflection equals to angle of incidence. We refer to such systems as \emph{convex billiards}.

\subsubsection{Billiard map in coordinates $(s,\theta)$}	
Each collision of the particle with $\partial D$ is described by the pair $(s,\theta)$, where $s$ denotes the arc-length parameter (mod 1) of the point of impact, and $\theta\in[0,\pi]$ is the angle between the incoming trajectory and the positively oriented tangent to $\partial D$ at $s$. 
Hence, the space of all possible collisions is the closed cylinder
\[
Q=S^1 \times [0,\pi]
\] 

The billiard map associated with the table $D$ on a surface $S$ is the map $\phi \colon Q\to Q$ that 
assigns to each collision $(s,\theta) \in Q$ the next collision 
\[
\phi(s,\theta)=(s_1(s,\theta),\theta_1(s,\theta)).
\] 

The map $\phi$ for general surfaces of constant curvature retains the same properties of the billiard map for planar convex billiards.

\begin{proposition}[\cite{DiasCarneiro:2024}] \label{prop:properties}
The map $\phi \colon Q \to Q$ satisfies the following properties:
\begin{enumerate} 
\item $\phi$ is a homeomorphism, 
\item $\operatorname{int} Q := S^1 \times (0,\pi)$ and $\partial Q:= S^1 \times \{0,\pi\}$ are $\phi$-invariant sets, 
\item $\phi|_{\operatorname{int} Q} \colon \operatorname{int} Q \to \operatorname{int} Q$ is a twist $C^1$ diffeomorphism,
\item $\phi|_{\partial Q} = \operatorname{id}_{\partial Q}$,
\item $\phi$ preserves the measure $d \bar{m} = \sin \theta \,ds \, d\theta$. 
\end{enumerate}
\end{proposition}

Let $\kappa(s)$ denote the geodesic curvature of $\partial D$ at the point with arc-length parameter $s$. Let $t(s,\theta)$ be the geodesic distance on $S$ between the points of $\partial D$ corresponding to $s$ and $s_1(s,\theta)$. Write 
\[
t = t(s,\theta), \quad \theta_1 = \theta_1(s,\theta), \quad \kappa = \kappa(s), \quad \kappa_1 = \kappa(s_1(s,\theta)).
\] 
Then for every $(s,\theta) \in \operatorname{int} Q$, the derivative $D \phi(s,\theta)$ is given by following expressions, corresponding respectively to the cases $S=\mathbb{E}^2$, $S=\mathbb{S}^2$ and $S=\mathbb{H}^2$~\cite{Chernov:2006we,Katok:1986tg,dosSantos:2017}:
\begin{equation} \label{eq:derivative}
\begin{gathered}
D \phi(s,\theta) = 
\begin{pmatrix}
	\frac{\kappa t -\sin \theta}{\sin \theta_1} & \frac{t}{\sin \theta_1} \\
	\frac{\kappa_1 \kappa t -\kappa_1 \sin \theta -\kappa \sin \theta_1}{\sin \theta_1} & \frac{\kappa_1 t -\sin \theta_1}{\sin \theta_1} 
\end{pmatrix}, \\
D\phi(s,\theta) = 
\begin{pmatrix}
\frac{\kappa \sin t - \cos t \sin\theta}{\sin\theta_1} & \frac{\sin t}{\sin\theta_1}\\
\frac{\sin t (\kappa \kappa_1 - \sin\theta \sin\theta_1) - \cos t (\kappa_1 \sin\theta + \kappa \sin\theta_1)}{\sin\theta_1} & \frac{\kappa_1 \sin t - \cos t \sin\theta_1}{\sin\theta_1}
\end{pmatrix}, \\
D\phi(s,\theta) = 
\begin{pmatrix}
\frac{\kappa \sinh t - \cosh t \sin\theta_1}{\sin\theta_1} & \frac{\sinh t}{\sin\theta_1} \\
\frac{\sinh t (\kappa \kappa_1 + \sin\theta \sin\theta_1) - \cosh t (\kappa_1 \sin\theta + \kappa \sin\theta_1)}{\sin\theta_1} & \frac{\kappa_1 \sinh t - \cosh t \sin\theta_1}{\sin\theta_1}
\end{pmatrix}.
\end{gathered}
\end{equation}

Moreover, in all cases, the derivative $D\phi$ admits the following limits (see~\cite{Katok:1986tg} and \cite[Proposition~3.4]{DiasCarneiro:2024}): for every $s \in S^1$, 
\begin{equation} \label{eq:zeroangle}
\lim_{\theta \to 0^+}  D \phi(s,\theta) = \lim_{\theta \to \pi^-}  D \phi(s,\theta) =
\begin{pmatrix}
	1 & \frac{2}{\kappa(s)} \\
	0 & 1
\end{pmatrix}.
\end{equation}

\subsubsection{Billiard map in coordinates $(s,r)$}
For the purposes of this work, it is convenient to replace the coordinate $\theta$ by $r=-\cos \theta \in [-1,1]$. In coordinates $(s,r)$, the set of all collisions is given by the closed cylinder $V = S^1 \times [-1,1]$, and the billiard map is denoted by $\Phi \colon V \to V$.

The map $\Phi$ inherits the properties of $\phi$ described in Proposition~\ref{prop:properties}. Namely: 1) $\Phi$ is a homeomorphism, 2) the sets $\operatorname{int} V := S^1 \times (-1, 1)$ and $\partial V := S^1 \times \{-1, 1\}$ are invariant under $\Phi$, 3) $\Phi|_{\operatorname{int} V}$ is a twist $C^1$ diffeomorphism, 4) $\Phi_{\partial V} = \operatorname{id}_{\partial V}$, 5) $\Phi$ preserves the measure $ds dr$. 

However, a key difference between the two maps is that $\det D\Phi \equiv 1$ on $\operatorname{int} V$, whereas for $\Phi$, one has  $\det \phi = \sin \theta / \sin \theta_1 \not\equiv 1$ on $\operatorname{int} Q$. This fact is the reason for working with the map $\Phi$ rather than with $\phi$.


\subsubsection{Billiard map on the $2$-torus} 
\label{su:torus}
We now introduce a map $T$ on the 2-torus, induced by the billiard map $\Phi$. Random perturbations of $T$ will be studied in Section~\ref{ss:random_billiards}.

Since the restriction of $\Phi$ to $\partial V$ is the identity, we can define an automorphism $T$ of the torus $\Tt^2 = \mathbb{R}/\mathbb{Z} \times \mathbb{R}/2 \mathbb{Z}$ endowed with the flat metric $ds^2+dr^2$ by choosing $R:=[0,1) \times [-1,1)$ as a fundamental domain of $\mathbb{R}/\mathbb{Z} \times \mathbb{R}/2 \mathbb{Z}$ and setting 
\[
T(s,r) = \Phi(s,r) \quad \text{for all } (s,r) \in R.
\] 

The properties of $\Phi$ imply the following properties of $T$: 1) $T\colon\mathbb T^2\to\mathbb T^2$ is a homeomorphism, 2) the sets $\operatorname{int}R=[0,1)\times(-1,1)$ and $\partial R=[0,1)\times\{-1\}$ are $T$-invariant, 3) the restriction $T|_{\operatorname{int}R}\colon\operatorname{int}R\to\operatorname{int}R$ is a $C^1$ diffeomorphism and $\det DT(s,r)=1$ for every $(s,r)\in\operatorname{int}R$, 4) every point of $\partial R$ is fixed by $T$, i.e. 
$T(s,-1)=(s,-1)$ for every $s\in[0,1)$, 5) $T$ preserves the normalized Haar measure $m$ on $\mathbb T^2$, given in the coordinates $(s,r)$ by $dm=c\,ds\,dr$,
where $c>0$ is chosen so that $m(\mathbb T^2)=1$.

\subsection{Circular billiards}
\label{su:circular}

Let $D\subset S$ be a geodesic disk. Since $S$ has constant curvature, the boundary $\partial D$ is necessarily a geodesic circle. In this paper, we restrict our attention to geodesic disks whose boundaries $\partial D$ have positive geodesic curvature.

Under these assumptions, the billiard map $\phi\colon Q\to Q$ in coordinates $(s,\theta)$ associated with the geodesic disk takes the simple form
\[
\phi(s,\theta) = (s+\alpha(\theta),\theta) \quad \text{for all } (s,\theta)\in Q,
\]
for a suitable smooth function $\alpha$ on $\operatorname{int} Q$; see~\cite{Bolotin1992IntegrableBilliards,santos2022break}. Similarly, the corresponding map $T\colon\Tt^2\to\Tt^2$ has the form
\[
T(s,\theta) = (s+\beta(r),r) \quad \text{for all } (s,\theta)\in\Tt^2,
\]
where $\beta(r)=\alpha(-\arccos r)$. Both maps $\phi$ and $T$ are integrable.

\subsection{Random additive perturbations of billiards}
\label{ss:random_billiards}

Let $T\colon\Tt^2\to\Tt^2$ be the billiard map associated with a convex domain $D$ on a surface of constant curvature. We consider random additive perturbations of $T$ in the framework of Section~\ref{sec:random maps}.

We equip $\Tt^2$ with its standard flat metric. Let $\nu$ be a probability measure on $\Tt^2$, and set $\Omega=(\Tt^2)^{\Nn}$ and $\rho_\nu=\nu^{\Nn}$. For each $x \in X$, define
\[
f\colon \Tt^2\to\Tt^2,
\qquad
f_x(y)=T(y)+x.
\]
The corresponding random dynamical system is given by the skew product
\[
F(\omega,y)=\bigl(\sigma(\omega),f_{\omega_0}(y)\bigr),
\qquad
(\omega,y)\in\Omega\times\Tt^2.
\]

\begin{lemma}
\label{lemma:0}
The billiard map $T\colon\Tt^2\to\Tt^2$ is a toral map in the sense of Definition~\ref{def:toral_map} with $d=2$ and $N:=\operatorname{int}R$. Consequently, for every probability measure $\nu$ on $\Tt^2$, the random dynamical system $F$ is a $\nu$-random additive perturbation of $T$.
\end{lemma}

\begin{proof}
Conditions~(1) and~(2) of Definition~\ref{def:toral_map} follow from the  properties of the map $T$. It remains to verify Condition~(3). Since $DT=D\Phi$ on $N=\operatorname{int}R$ and $m(\partial R)=0$, it suffices to prove that $\log^+\|D\Phi\|$ is integrable with respect to the measure $dsdr$ on $N$. Indeed, since $\Phi$ is a 2-dimensional map, $\det D\Phi \equiv 1$ implies that $\|(D\Phi)^{-1}\|=\|D\Phi\|$, and hence the integrability of $\log^+ \|D\Phi\|$ on $N$ gives the integrability of $\log^+ \|(D\Phi)^{-1}\|$ on $N$. 

Let  $h \colon [0,1) \times (0,\pi) \to N$ be the change of coordinates given by $h(s,\theta)= (s,-\cos \theta)$. Since $\Phi = h \circ \phi \circ h^{-1}$, the chain rule together with
\[
\|Dh(s,\theta)\|\leq 1, \qquad (s,\theta) \in [0,1) \times (0,\pi),
\]
and
\[
\|Dh^{-1}(s,r)\|=\frac{1}{\sqrt{1-r^2}}, \qquad (s,r) \in N,
\]
gives for every $(s,r) \in N$,
\[
\log^+\|D\Phi(s,r)\|
\leq
\log^+\|D\phi(h^{-1}(s,r))\|
-\frac12\log(1-r^2).
\]
The explicit formula for $D\phi$ and the boundedness of the free-flight time function $(s,\theta) \mapsto t(s,\theta)$ yield a constant $C\geq 1$ such that
\[
\|D\phi(h^{-1}(s,r))\|
\leq
\frac{C}{\sqrt{1-r_1(s,r)^2}},
\]
where $\Phi(s,r)=(s_1(s,r),r_1(s,r))$. Therefore for every $(s,r) \in N$,
\[
\log^+\|D\Phi(s,r)\|
\leq
\log C-\frac12\log(1-r^2)
-\frac12\log\bigl(1-r_1(s,r)^2\bigr).
\]
The function $r\mapsto-\log(1-r^2)$ is integrable on $(-1,1)$. Moreover, since $\Phi$ preserves $dsdr$,
\[
\int_N
-\log\bigl(1-r_1(s,r)^2\bigr)\,ds\,dr < \infty, \qquad 
\int_N
-\log(1-r^2)\,ds\,dr<\infty.
\]
Hence $\log^+\|D\Phi\|$ is integrable on $N$. Thus $T$ satisfies Condition~(3). The second assertion of the lemma follows from Lemma~\ref{le:add-pert}.
\end{proof}

\subsection{Vanishing exponents and circular billiards}
\label{sss:abscont}

For notational convenience, we denote the extremal Lyapunov exponents of $F$ by $\lambda^-$ and $\lambda^+$, omitting the subscript $F$. 


We now show that when the billiard table $D$ is a geodesic disk, the Lyapunov exponents $\lambda^-(\omega,y)$ and $\lambda^+(\omega,y)$ vanish for any choice of the probability measure $\nu$.

\begin{lemma}
\label{lemma: 8}
Suppose that $D$ is a geodesic disk. Then for every probability measure $\nu$ on $\Tt^2$, we have 
\[
\lambda^-(\omega,y) = \lambda^+(\omega, y) = 0 \quad \text{for }\rho_\nu \times m \text{-a.e. }(\omega, y) \in \Omega \times \Tt^2.
\]
\end{lemma}

\begin{proof} 
Let $\nu$ be a probability measure on $\Tt^2$. The expression of the billiard map $T$ associated with a geodesic disk $D$ is given in Section~\ref{su:circular}. For every $x\in \Tt^2$, we have
\[
Df_x(s,r)=DT(s,r)
=
\begin{pmatrix}
1 & \beta'(r)\\
0 & 1
\end{pmatrix},
\qquad (s,r)\in \operatorname{int} R.
\]
Hence
\[
Df_x(y)\frac{\partial}{\partial s}
=
\frac{\partial}{\partial s},
\qquad
(x,y)\in \Tt^2\times \operatorname{int} R.
\]
It follows that for $\rho_\nu\times m$-a.e. $(\omega,y)$, 
\[
\lim_{n\to+\infty}
\frac1n
\log\left\|
D F_\omega^n(y)\frac{\partial}{\partial s}
\right\|
=0,
\]
and hence one Lyapunov exponent vanishes $\rho_\nu\times m$-a.e. Since their sum is zero $\rho_\nu\times m$-a.e., both exponents vanish $\rho_\nu\times m$-a.e., as claimed.
\end{proof}

%




We now assume that $\nu \ll \operatorname{vol}$, and proceed to prove Theorem~\ref{thm:main}.

For such a $\nu$, all the conclusions of Section~\ref{Consequences} apply to the skew-product $F$. In particular, the probability measure $\rho_\nu \times m$ is ergodic, and the Lyapunov exponents $\lambda^-$ and $\lambda^+$ are constant for $\rho_\nu \times m$-almost every $(\omega, y) \in \Omega \times \Tt^2$. 

Recall that $\phi\colon Q\to Q$ is the billiard map in coordinates $(s,\theta) \in Q$ with $Q = S^1 \times [0,\pi]$.

\begin{proposition}
\label{prop:6}
Suppose $\lambda^+(\omega,y)=0$ for $\rho_\nu \times m$-a.e. $(\omega,y) \in \Omega \times \Tt^2$. Then there exists a strictly positive continuous function $\gamma \colon \operatorname{int} Q \to \mathbb{R}$ such that
\[
D \phi(s,\theta) \frac{\partial}{\partial s} = \gamma(s,\theta) \frac{\partial}{\partial s} \quad \text{for all } (s,\theta) \in \operatorname{int} Q.
\]
\end{proposition}

\begin{proof}
By Lemma~\ref{lemma:0}, Theorem~\ref{thm:invariant} applies to $F$. Moreover, the expression of $D\phi$ in~\eqref{eq:derivative} and the property~\eqref{eq:zeroangle}, together with the definitions of $\Phi$ and $T$, show that $DT$ admits a factorization required by Corollary~\ref{cor:K}, which therefore applies to $F$. Since $DT$ is continuous on $[0,1) \times (-1,1)$, it follows that $\{DT(s,r) \colon (s,r) \in [0,1) \times (-1,1)\}$ is contained in the support of $DT_* m$. Hence, by Theorem~\ref{thm:invariant} and Corollary~\ref{cor:K}, we obtain
\[
DT(s,r)_* \delta_{\widehat{e}} = \delta_{\widehat{e}} \quad \text{for all } (s,r) \in [0,1) \times (-1,1).
\]
 where $\widehat{e} \,$ is the element of $\mathbb{P}^1$ with homogeneous coordinates $[1:0]$. 
 
This implies that $\widehat{e} \,$ is a fixed point of the projective action of $DT(s,r)$ for all $(s,r) \in [0,1) \times (-1,1)$. Equivalently, the subspace $L_s$ spanned by the vector $\partial/\partial s$ is invariant with respect to $DT(s,r)$ for every $(s,r) \in [0,1) \times (-1,1)$.

In view of the definition of $T$, we see that $L_s$ is also invariant under $D\phi(s,\theta)$ for every $(s,\theta) \in \operatorname{int} Q$. Let $\langle \cdot,\cdot \rangle$ denote the Euclidean inner product on $\Rr^2$. Thus, we have
\begin{equation*} 
D\phi(s,\theta) \frac{\partial}{\partial s} = \gamma(s,\theta) \frac{\partial}{\partial s} \quad  \text{for every } (s,\theta) \in \operatorname{int} Q,
\end{equation*}
where $\gamma(s,\theta) := \langle \partial /\partial s,D\phi(s,\theta) \partial /\partial s \rangle$. 
Since $\phi$ is a $C^1$ diffeomorphism on $\operatorname{int} Q$, the function $\gamma$ is well-defined and continuous on $\operatorname{int} Q$. Moreover, $\gamma$ must be either strictly positive or strictly negative, because $\gamma(s,\theta)=0$ for some $(s,\theta) \in \operatorname{int} Q$ would imply $D\phi(s,\theta) \partial/\partial s = 0$, contradicting the invertibility of $D\phi(s,\theta)$. To establsh that $\gamma$ is strictly positive, it suffices to show that for fixed $s$, we have $\lim_{\theta \to 0^+} \gamma(s,\theta)$ exists and is positive. Indeed, by~\eqref{eq:zeroangle}, we have $\lim_{\theta \to 0^+} \gamma(s,\theta) = \lim_{\theta \to 0^+} \langle D\phi(s,\theta) \partial /\partial s,\partial / \partial s \rangle=1$.
\end{proof}

The following is Theorem~\ref{thm:main} written in a form slightly different.

\begin{theorem}
\label{thm:7}
If $D$ is not a geodesic disk, then $\lambda^+(\omega,y)>0$ for $(\rho_\nu\times m)$-a.e. $(\omega,y)\in\Omega\times\Tt^2$. Otherwise, $\lambda^+(\omega,y)=0$ for $(\rho_\nu\times m)$-a.e. $(\omega,y)\in\Omega\times\Tt^2$.
\end{theorem}

\begin{proof}
If $D$ is a geodesic disk, then $\lambda^+ = 0$ for $\rho_\nu \times m$-a.e. $(\omega,y) \in \Omega \times \Tt^2$ by Lemma~\ref{lemma: 8}.

Conversely, assume that $\lambda^+ = 0$ for $\rho_\nu \times m$-a.e. $(\omega,y) \in \Omega \times \Tt^2$. In the remainder of the proof, we first work with the billiard map $\Phi$ in coordinates $(s,r)$, and later return to the map $\phi$. 

Let $E$ be the sub-bundle of $TV$ defined by 
\[
E(s,r) = \operatorname{span} \left(\frac{\partial}{\partial s}\right) \quad \text{for every } (s,r) \in V.
\]
By Proposition~\ref{prop:6}, we have 
\[
D \Phi(s,r) E(s,r) = E(\Phi(s,r)) \quad \text{for all } (s,r) \in \operatorname{int} V. 
\]
Note that the integral curves of the sub-bundle $E$ are the circles $\Gamma_r$, corresponding to the level sets of the function $p(s,r):=r$. Since $E$ is $D \Phi$-invariant, the foliation $\{\Gamma_r\}$ is $\Phi$-invariant.

This implies that $\Phi$ must be a skew-product of the form $\Phi(s,r) = (b(s,r),a(r))$ for some maps $a \colon [-1,1] \to [-1,1]$ and $b \colon V \to S^1$ such that $a$ is a homeomorphism fixing the points $r=-1$ and $r=1$ and a $C^1$ diffeomorphism on $(-1,1)$, whereas $b$ is continuous and $C^1$ on the interior of $V$. Since $\Phi$ preserves $ds dr$ and $p \circ \Phi = a \circ p$, the measure $dr = p_* ds dr$ must be $a$-invariant. This together with the fact that $r=-1$ and $r=1$ are fixed points of $a$ implies that $a$ is the identity. Thus $\Phi(s,r)=(b(s,r),r)$. It follows that each circle $\Gamma_r$ is $\Phi$-invariant.  
	
The same is true for the map $\phi$: it is a skew product of the form $\phi(s,\theta)=(\bar{b}(s,\theta),\theta)$ for some $C^1$ function $\bar{b}$. Note also that in this case $\det D \phi=\sin \theta/\sin \theta_1 = 1$. As a consequence, the entries on the main diagonal of the matrix of $D \phi$ must be equal to 1. Comparing with the expression of $D\phi$ in~\eqref{eq:derivative}, we 
conclude that the curvature $k$ of $\partial D$ is constant. Hence, $D$ is a geodesic disk.

We have proved that, if $\lambda^+=0$ for $(\rho_\nu\times m)$-a.e. $(\omega,y)$, then $D$ is a geodesic disk. Hence, if $D$ is not a geodesic disk, then $\lambda^+>0$ on a set of positive $(\rho_\nu\times m)$-measure. Since $\rho_\nu\times m$ is ergodic for the skew-product $F$ by Corollary~\ref{cor:skew-product-ergodicity}, it follows that $\lambda^+>0$ for $(\rho_\nu\times m)$-a.e. $(\omega,y)$. This completes the proof.
\end{proof}

\begin{remark}
The second part of the proof of Theorem~\ref{thm:7} can alternatively be derived from a result of Bialy~\cite{Bialy:1993ty,Bialy:2013}, which characterizes circular billiard tables among all convex tables on surfaces of constant curvature. Specifically, Bialy proved that circular billiards are the only ones for which every trajectory has no conjugate points. Moreover, he showed that the absence of conjugate points is equivalent to the existence of a measurable monotone invariant sub-bundle. By Proposition~\ref{prop:6}, the sub-bundle $E$ introduced in the proof of Theorem~\ref{thm:7} has this property. Thus, Bialy's result applies, and we may conclude that $D$ is a geodesic disk. However, our approach is more direct, since it relies on the specific form of the invariant sub-bundle $E$ derived in Proposition~\ref{prop:6}. For an alternative proof of Bialy's theorem in the planar case, see also~\cite{Wojtkowski:1994tj}.
\end{remark}

\begin{remark}
We discuss here a possible extension of Theorem~\ref{thm:7} to magnetic billiards. In these systems, a charged point particle moves under the influence of a constant magnetic field. Unlike standard billiards, where the particle travels along geodesics between elastic reflections at the boundary, the motion in magnetic billiards is governed by the Lorentz force, so that trajectories are curves of constant geodesic curvature determined by the field strength;  for instance, see~\cite{Berglund1996MagneticBilliards,Gutkin2001}. Upon collision with the boundary, the particle undergoes specular reflection.

For convex planar tables, Berglund and Kunz~\cite{Berglund1996MagneticBilliards} showed that the magnetic billiard map shares many qualitative features of the corresponding non-magnetic billiard map, with one notable exception: among the two boundary components of the phase space, only the circle $S^{1}\times\{-1\}$ is invariant under the magnetic billiard map, while the other component is not invariant. Nevertheless, the presence of this invariant boundary component suggests that an analog of Theorem~\ref{thm:main} may hold for magnetic billiards with convex planar tables.
\end{remark}



\section{Random generalized standard maps}
\label{sec:random_standard}

This section deals with random additive perturbations of generalized standard maps. Theorem~\ref{thm:standard} follows from the two results proved below, Proposition~\ref{prop:standard-absolute}, which treats absolutely continuous noise, and Proposition~\ref{prop:standard-singular}, which treats the case of a vertical line noise. 

Write $y=(y_1,y_2) \in \Tt^2$. Recall that a generalized \emph{standard map} is a map $g_V:\Tt^2\to\Tt^2$ defined by
\[
g_V(y_1,y_2)
=
\left(y_1+y_2+V(y_1),\, y_2+V(y_1)\right) \bmod 1,
\]
where $V:\Tt^1\to\Rr$ is a $C^1$ function. Such maps are diffeomorphisms that preserve the normalized Haar measure $m$ on $\Tt^2$. Indeed, 
\[
Dg_V(y_1,y_2)
=
\begin{pmatrix}
1+V'(y_1) & 1\\
V'(y_1) & 1
\end{pmatrix},
\]
and hence $\det Dg_V(y_1,y_2)=1$. If $V\equiv c$ is constant, then $g_V$ is a toral affine skew-product map:
\[
g_V(y_1,y_2)
=
(y_1+y_2+c,\, y_2+c) \bmod 1.
\]

Since $g_V$ is a diffeomorphism and preserves $m$, it is a toral map in the sense of Definition~\ref{def:toral_map} with $d=2$ and $N:=\Tt^2$.

In what follows, we use the notation of Subsection~\ref{random additive} with $d=2$. Thus $F \colon \Omega \times M \to \Omega \times M$ denotes the $\nu$-random additive perturbation of $g_V$, and $\lambda^{+}(\omega,y)=\lambda^{+}_F(\omega,y)$ denotes its maximal Lyapunov exponent. Unless otherwise specified, almost everywhere always means with respect to $\rho_\nu \times m$. 

\begin{proposition}
\label{prop:standard-absolute}
Suppose that $\nu \ll m$. If $V$ is constant, then $\lambda^+(\omega,y) = 0$ for $(\rho_\nu \times m)$-a.e. $(\omega,y)$. Otherwise, $\lambda^+(\omega,y) > 0$ for $(\rho_\nu \times m)$-a.e. $(\omega,y)$.
\end{proposition}

\begin{proof}
By Proposition~\ref{prop:unique-ergodicity-additive}, the measure $\rho_\nu \times m$ is ergodic. Therefore there exists a constant $b \ge 0$ such that $\lambda^+ = b$ a.e.

Assume first that $V$ is not constant. Then $V'$ takes both positive and negative values. Consequently, $\operatorname{tr}Dg_V=2+V'$ is larger than $2$ at some points, and belongs to $(0,2)$ at others. Corollary~\ref{cor:simple} implies that $b>0$. 

If $V$ is constant, then
\[
Dg_V \equiv
\begin{pmatrix}
1 & 1 \\
0 & 1
\end{pmatrix}.
\]
It follows that the derivative cocycle has polynomial growth, and therefore $b=0$. This proves the proposition.
\end{proof}


\begin{proposition}
\label{prop:standard-singular}
Suppose that $\nu$ is a vertical line measure. If $V$ is constant, then $\lambda^+(\omega,y) = 0$ for $(\rho_\nu \times m)$-a.e. $(\omega,y)$. Otherwise, $\lambda^+(\omega,y) > 0$ for $(\rho_\nu \times m)$-a.e. $(\omega,y)$.
\end{proposition}

\begin{proof}
Recall that $F \colon \Omega\times\Tt^2\to\Omega\times\Tt^2$ is the $\nu$-random additive perturbation of $g_V$. We reduce the proof to the absolutely continuous case by grouping the noise variables $\omega_n$ in pairs.

Recall that $\tau_x$ denotes the translation on $\Tt^2$ by $x \in \Tt^2$, and that $f_x = \tau_x \circ g_V$. A direct computation gives
\[
g_V \circ \tau_{(0,s)} = \tau_{(s,s)} \circ g_V, \quad s \in \Tt^1.
\]
As a consequence, for all $s,t \in \Tt^1$, we have
\begin{equation} \label{eq:two-steps}
f_{(0,t)} \circ f_{(0,s)} = \tau_{(0,t)} \circ g_V \circ \tau_{(0,s)} \circ g_V  = \tau_{(s,s+t)} \circ g^2_V.
\end{equation}

Write $x=(x_1,x_2)$ and $y=(y_1,y_2)$. Define the map $\phi \colon \Tt^2\times\Tt^2\to\Tt^2$ by
\[
\phi(x,y)=(x_1+x_2+y_1,x_2+y_2), \quad x,y \in \Tt^2.
\]
Since $\supp \nu \subset {0}\times \Tt^1$, if $x=(0,s)$ and $y=(0,t)$, then $\phi(x,y)=(s,s+t)$. Therefore, for $(\nu \times \nu)$-a.e. $(\omega_0,\omega_1) \in \Tt^2 \times \Tt^2$, we have
\[
f_{\omega_1} \circ f_{\omega_0} = \tau_{\phi(\omega_0,\omega_1)} \circ g^2_V.
\]

Let $\nu'=\phi_*(\nu\times\nu)$. It is straightforward to check that since $\nu$ is a vertical line measure,  $d\nu' = h(s)h(t-s)dsdt$ for $(s,t) \in \Tt^1 \times \Tt^1$. In particular, $\nu' \ll m$. 

The same observation shows that $\phi$ induces an isomorphism mod $0$ between the probability spaces $(\Tt^2\times\Tt^2,\BB\times\BB,\nu\times\nu)$ and $(\Tt^2,\BB,\nu')$, where $\BB$ is the Borel $\sigma$-algebra of $\Tt^2$.

Define a map $\Phi \colon \Omega \to \Omega$ by
\[
(\Phi(\omega))n =\phi(\omega_{2n},\omega_{2n+1}),
\quad
\omega \in \Omega,\ n \in \Nn.
\]
Let $\rho_{\nu'}=(\nu')^{\Nn}$. By the previous observations on $\phi$, it follows that $\Phi$ is an isomorphism mod $0$ between the Bernoulli shifts $(\Omega,\Sigma,\rho_\nu,\sigma^2)$ and $(\Omega,\Sigma,\rho_{\nu'},\sigma)$, i.e. $\Phi$ is an isomorphism mod $0$ between the probability spaces $(\Omega,\Sigma,\rho_\nu)$ and $(\Omega,\Sigma,\rho_{\nu'})$ such that
\[
\Phi \circ \sigma^2(\omega) = \sigma \circ \Phi(\omega) \quad \text{for } \rho_\nu \text{-a.e. } \omega \in \Omega.
\]

Let $G \colon \Omega \times \Tt^2 \to \Omega \times \Tt^2$ be the $\nu'$-random additive perturbation of $g^2_V$. Define $H \colon \Omega \times \Tt^2 \to \Omega \times \Tt^2$ by
\[
H(\omega,y)=(\Phi(\omega),y), \quad (\omega,y) \in \Omega \times \Tt^2.
\]
Then $H_*(\rho_\nu\times m)=\rho_{\nu'}\times m$, and $H$ conjugates $F^2$ to $G$ mod $0$. Indeed, from~\eqref{eq:two-steps}, it follows that for $\rho_\nu$-a.e. $\omega\in\Omega$ and every $n \ge 1$,
\begin{align}
\label{eq:identity}
G^n_{\Phi(\omega)} & =\tau_{\phi(\omega_{2n-2},\omega_{2n-1})}\circ g_V^{,2}\circ\cdots\circ
\tau_{\phi(\omega_0,\omega_1)}\circ g_V^{,2} \\
&=\tau_{\omega_{2n-1}}\circ g_V\circ\tau_{\omega_{2n-2}}\circ g_V\circ\cdots\circ
\tau_{\omega_1}\circ g_V\circ\tau_{\omega_0}\circ g_V = F^{2n}_\omega.  \notag
\end{align}

Denote by $\lambda^+_G$ and $\lambda^+_F$ the maximal Lyapunov exponents of $G$ and $F$, respectively. For $(\rho_\nu \times m)$-a.e. $(\omega,y) \in \Omega \times \Tt^2$, identity~\eqref{eq:identity} gives
\begin{align*}
\lambda^+_G\left(\Phi(\omega),y\right) & = \lim_{n \to +\infty} \frac{1}{n} \log \|DG^n{\Phi(\omega)}(y)\| \\
& =  \lim_{n \to +\infty} \frac{1}{n} \log  \|DF^{2n}{\omega}(y)\| =  2 \lambda^+_F(\omega,y).
\end{align*}
Since $\Phi_*\rho_\nu=\rho_{\nu'}$, it follows that $\lambda^+_F$ vanishes (is positive) $\rho_\nu\times m$-a.e. if and only if $\lambda^+_G$ vanishes (is positive) $\rho_{\nu'}\times m$-a.e.

We now show that Proposition~\ref{prop:standard-absolute} applies to the system $G$. The measure $\nu'$ is absolutely continuous with respect to $m$. Hence, we only need to verify the trace condition for $g^2_V$. A direct computation yields
\[
\operatorname{tr} Dg^2_V(y_1,y_2) = 2+2V'(y_1) + V'\bigl(y_1+y_2+V(y_1)\bigr)\bigl(2+V'(y_1)\bigr).
\]
Let $\bar y_1\in\Tt^1$ be such that $V'(\bar y_1)=0$. Then
\[
\operatorname{tr} Dg^2_V(\bar y_1,y_2) = 2+2V'\bigl(\bar y_1+y_2+V(\bar y_1)\bigr).
\]
As $y_2$ varies in $\Tt^1$, the argument $\bar y_1+y_2+V(\bar y_1)$ ranges over all of $\Tt^1$. If $V$ is not constant, then $V'$ takes both positive and negative values and vanishes somewhere. Thus the trace of $Dg^2_V$ takes values larger than $2$ and values in $(0,2)$, and the hypotheses of Corollary~\ref{cor:simple} are satisfied.

Therefore the proof of Proposition~\ref{prop:standard-absolute}, applied to $g_V^2$ and $\nu'$, shows that $\lambda_G^+$ is a.e. constant, with zero value if and only if $V$ is constant. The relation between $\lambda_G^+$ and $\lambda_F^+$, then gives the desired conclusion for $\lambda_F^+$.

\end{proof}









\section*{Acknowledgements}

G. Del Magno acknowledges support from the MIUR Excellence Department Project (CUP I57G22000700001) awarded to the Department of Mathematics, University of Pisa, and from the PRIN Project 2022NTKXCX \textit{``Stochastic properties of dynamical systems''}, funded by the Italian Ministry of University and Research.

JLD and JPG were funded by national funds through FCT – Fundação para a Ciência e a Tecnologia, I.P., in the framework of the unit UID/06522/2025.


\bibliographystyle{abbrv}
\bibliography{rfrncs}
\end{document}